\documentclass[reqno,11pt]{article}
\usepackage{latexsym,amsmath,amscd,amssymb,graphics,mathrsfs}
\usepackage{amsthm}
\usepackage{enumerate}
\usepackage{color}
\usepackage{url}
\usepackage[all]{xy}
\usepackage{enumerate}
\usepackage{hyperref}

\newcommand{\rem}[1]{}
\makeatletter

\newcommand \al{\alpha}
\newcommand\be{\beta}
\newcommand\ga{\gamma}

\newcommand\ep{\varepsilon}
\newcommand\ze{\zeta}

\renewcommand\th{\theta}
\newcommand\io{\iota}

\newcommand\la{\lambda}

\newcommand\si{\sigma}

\newcommand\ph{\varphi}

\newcommand\om{\omega}
\newcommand\Ga{\Gamma}

\newcommand\La{\Lambda}
\newcommand\Si{\Sigma}

\newcommand\Om{\Omega}

\newcommand\resp{resp.\ }
\newcommand\ie{i.e.\ }

\newcommand\oo{{\infty}}

\renewcommand\o{\circ}

\newcommand\x{\times}
\newcommand\on{\operatorname}

\newcommand\Emb{\on{Emb}}

\newcommand\Den{\on{Den}}

\newcommand\iso{\on{iso}}

\newcommand\N{\mathcal{N}}
\newcommand\D{\mathcal{D}}

\newcommand\Diff{\on{Diff}}

\newcommand\id{\on{id}}

\newcommand{\symp}{\on{symp}}

\newcommand\pr{\on{pr}}

\renewcommand\L{\on{pounds}}
\newcommand\Ham{\on{Ham}}
\newcommand\Symp{\on{Symp}}
\newcommand\Isom{\on{Isom}}
\newcommand\KKS{\on{KKS}}
\newcommand\Flag{\on{Flag}}
\newcommand\pFlag{\on{Fr}}

\newcommand\ham{\on{\mathfrak{ham}}}

\newcommand\SL{\on{GL}}
\newcommand\g{\mathfrak g}

\newcommand\Gr{\on{Gr}}

\newcommand\ZZ{\mathbb Z}
\newcommand\TT{\mathbb T}
\newcommand\RR{\mathbb{R}}\newcommand\R{\mathbb{R}}
\newcommand\PP{\mathbb{P}}

\newcommand\X{\mathfrak X}
\renewcommand\S{\mathcal S}
\newcommand\E{\mathcal E}
\renewcommand\O{\mathcal O}
\newcommand\F{\mathcal F}
\newcommand\G{\mathcal G}

\renewcommand\L{\mathcal L}

\newcommand\wt{\on{wt}}

\newcommand\hwt{\on{hwt}}

\DeclareMathOperator{\evol}{evol}
\DeclareMathOperator{\Evol}{Evol}

\definecolor{burgund}{RGB}{153,0,51}      %burgundy

\begin{document}

\newtheorem{theorem}{Theorem}[section]
\newtheorem{definition}[theorem]{Definition}
\newtheorem{lemma}[theorem]{Lemma}
\newtheorem{proposition}[theorem]{Proposition}
\newtheorem{corollary}[theorem]{Corollary}
\theoremstyle{remark}
\newtheorem{remark}[theorem]{Remark}
\newtheorem{example}[theorem]{Example}

%%%%%%%%

\title{Weighted nonlinear flag manifolds as coadjoint orbits}
\author{Stefan Haller$^{1}$ and Cornelia Vizman$^{2}$ }

\addtocounter{footnote}{1}
\footnotetext{Department of Mathematics, University of Vienna, Oskar-Morgenstern-Platz 1, 1090 Vienna, Austria.
\texttt{stefan.haller@univie.ac.at}
\addtocounter{footnote}{1}}

\footnotetext{Department of Mathematics,
West University of Timi\c soara, Bd. V. Parvan 4,
300223 Timi\c soara, Romania.
\texttt{cornelia.vizman@e-uvt.ro}
\addtocounter{footnote}{1} }

\addtocounter{footnote}{1}

\date{ }
\maketitle
\makeatother

\begin{abstract}
A weighted nonlinear flag is a nested set of closed submanifolds,
each submanifold endowed with a volume density.
We study the geometry of Fr\'echet manifolds of weighted nonlinear flags,
in this way generalizing the weighted nonlinear Grassmannians.
When the ambient manifold is symplectic, we use these nonlinear flags to describe a class of coadjoint orbits of the group of Hamiltonian diffeomorphisms,
orbits that consist of weighted isotropic nonlinear flags.
%These coadjoint orbits can be obtained also via reduction at zero within a dual pair that generalizes the Marsden-Weinstein ideal fluid dual pair.
\end{abstract}
%\tableofcontents

\section{Introduction}

In this article we study manifolds of weighted nonlinear flags,
motivated by the fact  that one can use them to describe new coadjoint orbits of the Hamiltonian group.
This adds to the already known coadjoint orbits described with configuration spaces of points, weighted isotropic nonlinear Grassmannians  \cite{W90,Lee,GBV17},
symplectic nonlinear Grassmannians \cite{HV}, and manifolds of symplectic nonlinear flags \cite{HV20}.

Let $M$ be a smooth manifold, and let $\S=(S_1,\dotsc,S_r)$ be a collection of closed manifolds of strictly increasing dimensions. 
We consider the Fr\'echet manifold $\Flag_{\S}(M)$ of nonlinear flags of type $\S$, 
i.e., sequences of nested embedded submanifolds $N_1\subseteq\cdots\subseteq N_r$ in $M$, with $N_i$ diffeomorphic to $S_i$.
Considering submanifolds equipped with nowhere zero densities, one obtains the manifold $\Flag_{\S}^{\wt}(M)$ of weighted nonlinear flags.
We describe its Fr\'echet manifold structure in two ways:
as a splitting smooth submanifold of the cartesian product of weighted nonlinear Grassmannians of type $S_i$ in $M$,
and as a locally trivial smooth fiber bundle over $\Flag_{\S}(M)$ associated to the principal bundle of nonlinear frames of type $\S$ in $M$.
To each weighted nonlinear flag one associates a compactly supported distribution on $M$ with controlled singularities, by the $\Diff(M)$ equivariant inclusion
\begin{equation*}\label{E:JJ}
J:\Flag^{\wt}_{\S}(M)\hookrightarrow C^\infty(M)^*,\quad
\bigl\langle J((N_1,\nu_1),\dotsc,(N_r,\nu_r)),f\bigr\rangle:=\sum_{i=1}^r\int_{N_i}f\nu_i.
\end{equation*}
%\todo{One of our main results concerns the $\Diff(M)$ orbits in $\Flag_{\S}^{\wt}(M)$. These are submanifolds ...}
The $\Diff(M)$ orbits in $\Flag_{\S}^{\wt}(M)$ are submanifolds of finite codimension,
for which we give an explicit homological description, up to connected components,
using nonlinear flags decorated with cohomology classes (see Theorems~\ref{T:wtFlag.homog} and~\ref{L4}).

Inspired by the results in \cite{W90,Lee,GBV17} on weighted isotropic nonlinear Grassmannians,
we consider the manifold of weighted isotropic nonlinear flags in a symplectic manifold $(M,\om)$.
The orbits for the natural action of the Hamiltonian group $\Ham_c(M)$ are submanifolds of finite codimension,
described as leaves of an isodrastic foliation.
%\todo{described as $N_r\in\L$ plus decoration?}
Each isodrastic leaf of weighted nonlinear flags comes equipped with a canonical weakly non-degenerate symplectic form, and the map $J$ restricts to an equivariant moment map for the $\Ham_c(M)$ action, thus identifying the leaf with a coadjoint orbit of the Hamiltonian group (see Theorem~\ref{t2}).
Moreover, this coadjoint orbit is a splitting symplectic submanifold in a product of coadjoint orbits of weighted submanifolds of type $S_i$ in $M$.

The  lowest dimensional examples are the coadjoint orbits of $\Ham_c(\RR^2)$
consisting of pointed weighted vortex loops, treated in \cite{CVpre}.
We give more examples with nested spheres or tori, and provide explicit descriptions of the corresponding coadjoint orbits of the Hamiltonian group.

\paragraph{Acknowledgments.}
The first author would like to thank the West University of Ti\-mi\-\c soa\-ra for the warm hospitality.
He gratefully acknowledges the support of the Austrian Science Fund (FWF) grant P31663.
The second author would like to thank the University of Vienna for the warm hospitality.
She was partially supported by CNCS UEFISCDI, project number PN-III-P4-ID-PCE-2020-2888.

%%%%%%%%%%%%%%%%%%%%%%%%%

\section{Manifolds of weighted nonlinear flags}\label{s3}

A nonlinear flag is a sequence of nested closed submanifolds $N_1\subseteq\cdots\subseteq N_r$ in a smooth manifold $M$.
A weighted nonlinear flag is a nonlinear flag together with a volume density $\nu_i$ on each submanifold $N_i$.
Integrating against test functions $f\in C^\infty(M)$, a weighted nonlinear flag provides a compactly supported distribution on $M$ with mild singularities, $\sum_{i=1}^r\int_{N_i}f\nu_i$.

We will show that the space of all weighted nonlinear flags in $M$ is a Fr\'echet manifold in a natural way.
In fact, this is the total space of a locally trivial smooth bundle over the manifold of nonlinear flags discussed in~\cite{HV20}.
The natural $\Diff(M)$ action on the base of this bundle is locally transitive \cite[Proposition~2.9(a)]{HV20}.
The main aim of this section is to describe the $\Diff(M)$ orbits in the space of weighted nonlinear flags, see Theorem~\ref{T:wtFlag.homog} below.

\subsection{Weighted nonlinear Grassmannians}\label{2.2}

In this section we recall some basic facts about the manifolds of weighted submanifolds that appear in \cite{W90,Lee,GBV17}.
These weighted nonlinear Grassmannians constitute a special case of the weighted nonlinear flags to be introduced in Section~\ref{pre}.
We present them here in a manner that readily generalizes to the setting of nonlinear flags.

Let $S$ be a closed manifold of dimension $k$, allowed to be nonconnected and nonorientable.
For each manifold $M$, we let $\Gr_S(M)$ denote the \emph{nonlinear Grassmannian of type $S$ in $M$,} i.e., the space of all smooth submanifolds in $M$ that are diffeomorphic to $S$.
Moreover, we let $\Emb_S(M)$ denote the space of all parametrized submanifolds of type $S$ in $M$, i.e., the space of all smooth embeddings of $S$ into $M$.
Both, $\Emb_S(M)$ and $\Gr_S(M)$, are Fr\'echet manifolds in a natural way.
Furthermore, the $\Diff(M)$ equivariant map 
\begin{equation}\label{E:frame}
	\Emb_S(M)\to\Gr_S(M),\qquad\varphi\mapsto\varphi(S),
\end{equation}
is a smooth principal bundle with structure group $\Diff(S)$, a Fr\'echet Lie group, see \cite{BF81,M80b,M80c} and \cite[Theorem~44.1]{KM}.

For each closed $k$-dimensional manifold $S$, let 
\[
	\Den(S):=\Gamma^\infty(|\Lambda|_S)=\Omega^{k}(S;\O_S)
\]
denote the space of all smooth densities on $S$.
Here $\O_S$ denotes the orientation bundle of $S$ and $|\La|_S=\La^{k}T^*S\otimes \O_S$, see e.g.~\cite{Leebook}.
Densities on $S$ are the geometric quantities which can be integrated over $S$ in a coordinate independent way, without specifying an orientation or even assuming orientability.
We denote by $\Den_\x(S)$ the space of \emph{volume densities}, i.e., the space of nowhere vanishing densities.
Clearly, this is an open subset in the Fr\'echet space $\Den(S)$.

We define the \emph{weighted nonlinear Grassmannian of type $S$ in $M$} by
\begin{equation}\label{equiv}
	\Gr^{\wt}_{S}(M):=\bigl\{(N,\nu)\bigm|N\in\Gr_S(M),\nu\in\Den_\x(N)\bigr\},
\end{equation}
that is, the space of all submanifolds of type $S$ in $M$, decorated with a nowhere zero density.
We equip this space with the structure of a Fr\'echet manifold by declaring the natural bijection
\begin{equation}\label{grsw}
	\Gr_S^{\wt}(M)=\Emb_S(M)\x_{\Diff(S)}\Den_\x(S),\qquad\bigl(\ph(S),\ph_*\mu\bigr)\leftrightarrow[\ph,\mu],
\end{equation} 
to be a diffeomorphism.
Here the right hand side denotes the total space of the bundle associated to the nonlinear frame bundle in \eqref{E:frame} and the natural $\Diff(S)$ action on $\Den_\x(S)$.
In particular, the canonical forgetful map 
\begin{equation}\label{nop}
	\Gr^{\wt}_S(M)\to\Gr_S(M),\qquad(N,\nu)\mapsto N,
\end{equation}
becomes a locally trivial smooth bundle with typical fiber $\Den_\x(S)$.
Indeed, it corresponds to the bundle projection of the associated bundle $\Emb_S(M)\x_{\Diff(S)}\Den_\x(S)\to\Gr_S(M)$ via the identification in \eqref{grsw}.

There is a canonical $\Diff(M)$ equivariant map
\[
	J:\Gr^{\wt}_S(M)\to C^\infty(M)^*,\qquad\langle J(N,\nu),f\rangle:=\int_Nf\nu.
\]
This map is injective, and its image consists of compactly supported distributions with mild singularities: $J(N,\nu)$ is supported on $N$, and its wave front set coincides with the conormal bundle of $N$.

Let $\mu\in\Den_\x(S)$ be a volume density.
The space
\[
	\Gr_{S,\mu}^{\wt}(M):=\bigl\{(N,\nu)\in\Gr^{\wt}_S(M)\bigm|(S,\mu)\cong(N,\nu)\bigr\}
\]
is called the \emph{nonlinear Grassmannian of weighted submanifolds of type $(S,\mu)$} in $M$. 
It consists of all weighted submanifolds $(N,\nu)$ in $M$ such that there exists a diffeomorphism $S\to N$ taking $\mu$ to $\nu$.
Denoting the $\Diff(S)$ orbit of $\mu$ by $\Den(S)_\mu$, the identification in~\eqref{grsw} restricts to a canonical bijection
\begin{equation}\label{E:grswmu}
	\Gr_{S,\mu}^{\wt}(M)=\Emb_S(M)\times_{\Diff(S)}\Den(S)_\mu.
\end{equation}
It is well known \cite{M65} that the $\Diff(S)_0$ orbit of $\mu$ is a convex subset that consists of all volume densities on $S$ that represent the same cohomology class as $\mu$ in $H^k(S;\O_S)$, the de~Rham cohomology with coefficients in the orientation bundle.
Hence, the $\Diff(S)$ orbit of $\mu$ coincides with the set of all volume densities on $S$ that are in the preimage of $H^k(S;\mathcal O_S)_{[\mu]}$, the (finite) $\Diff(S)$ orbit of $[\mu]$ in $H^k(S;\O_S)$, under the $\Diff(S)$ equivariant linear map 
\begin{equation}\label{E:h}
	h_S:\Den(S)\to H^k(S;\O_S),\qquad h_S(\al)=[\al].
\end{equation}
More succinctly,
\begin{equation}\label{E:DenSmu}
	\Den(S)_\mu=\Den_\x(S)\cap h_S^{-1}\bigl(H^k(S;\mathcal O_S)_{[\mu]}\bigr).
\end{equation}
Hence, $\Den(S)_\mu$ is an open subset in a finite union of parallel closed affine subspaces with finite codimension.
In particular, $\Den(S)_\mu$ is a splitting smooth submanifold in $\Den_\x(S)$ with finite codimension $\dim H^k(S;\mathcal O_S)$ and with tangent spaces
\begin{equation}\label{talf}
	T_\al\Den(S)_\mu=\ker h_S=d\Om^{k-1}(S;\mathcal O_S).
\end{equation}
Using \eqref{E:grswmu} we conclude that $\Gr^{\wt}_{S,\mu}(M)$ is a splitting smooth submanifold in $\Gr^{\wt}_S(M)$ with finite codimension $\dim H^k(S;\mathcal O_S)$.
Moreover, the canonical forgetful map in \eqref{nop} restricts to a locally trivial smooth fiber bundle $\Gr^{\wt}_{S,\mu}(M)\to\Gr_S(M)$ with typical fiber $\Den(S)_\mu$.

The space of \emph{cohomologically weighted submanifolds of type $S$ in $M$} is defined as
\[
	\Gr_{S}^{\hwt}(M):=\bigl\{(N,[\nu]):N\in\Gr_S(M),[\nu]\in H^k(N;\mathcal O_N)\bigr\}.
\]
Using the canonical bijection
\begin{equation}\label{E:GrShwt}
	\Gr_{S}^{\hwt}(M)=\Emb_S(M)\x_{\Diff(S)}H^k(S;\O_S),
\end{equation}
we turn $\Gr_{S}^{\hwt}(M)$ into a smooth vector bundle of finite rank $\dim H^k(S;\O_S)$ over $\Gr_S(M)$. 
The canonical $\Diff(M)$ equivariant map
\[
	h_{\Gr_S(M)}:\Gr_S^{\wt}(M)\to\Gr_S^{\hwt}(M),\quad (N,\nu)\mapsto(N,[\nu]),
\]
is a smooth bundle map over $\Gr_S(M)$.
Indeed, via the diffeomorphisms in \eqref{E:grswmu} and \eqref{E:GrShwt} it corresponds to the map induced by \eqref{E:h}.

The space of \emph{cohomologically weighted submanifolds of type $(S,[\mu])$ in $M$} is defined by
\[
	\Gr_{S,[\mu]}^{\hwt}(M):=\bigl\{(N,[\nu])\in\Gr^{\hwt}_S(M):(N,[\nu])\cong(S,[\mu])\bigr\}
\]
and consists of all cohomologically weighted submanifolds $(N,[\nu])$ such that there exists a diffeomorphism $S\to N$ taking the cohomology class $[\mu]$ to $[\nu]$.
As \eqref{E:GrShwt} restricts to a bijection
\[
	\Gr_{S,[\mu]}^{\hwt}(M)=\Emb_S(M)\x_{\Diff(S)}H^k(S;\O_S)_{[\mu]},
\] 
we see that $\Gr_{S,[\mu]}^{\hwt}(M)$ is a finite covering of $\Gr_S(M)$. 
Using \eqref{E:DenSmu} we conclude
\begin{equation}\label{E:GrSmu.level}
	\Gr_{S,\mu}^{\wt}(M)=h_{\Gr_S(M)}^{-1}\bigl(\Gr_{S,[\mu]}^{\hwt}(M)\bigr).
\end{equation}

It is well known that the $\Diff_c(M)$ action on $\Emb_S(M)$ admits local smooth sections, see for instance \cite[Lemma~2.1(c)]{HV20}.
Furthermore, the (transitive) $\Diff(S)$ action on $\Den(S)_\mu$ also admits local smooth sections.
The latter can be shown using Moser's method of proof in \cite[Section~4]{M65}, see Lemma~\ref{L:diffmuS} below.
Using Lemma~\ref{before.sec} in the Appendix, we conclude that the natural $\Diff_c(M)$ action on $\Gr^{\wt}_{S,\mu}(M)$ admits local smooth sections.
In particular, this action is locally transitive.
Hence, each connected component of $\Gr^{\wt}_{S,\mu}(M)$ is a $\Diff_c(M)_0$ orbit.
Consequently, each $\Diff_c(M)$ or $\Diff(M)$ orbit in $\Gr^{\wt}_{S,\mu}(M)$ is a union of connected components.

\begin{remark}
	Poincar\'e duality provides a canonical $\Diff(S)$ equivariant isomorphism 
	\[
		H^k(S;\O_S)=H_0(S;\RR).
	\]
	Hence, specifying a cohomology class $[\mu]\in H^k(S;\O_S)$ amounts to specifying the total volume of $\mu$ on each connected component of $S$.
\end{remark}

\begin{example}
	If $S$ is connected, then $H^k(S;\O_S)=\R$ and the $\Diff(S)$ action is trivial on this cohomology.
	Hence, the orbit $H^k(S;\O_S)_{[\mu]}$ is a one-point set, and 
	\begin{equation}\label{dens}
		\Den(S)_\mu=\bigl\{\al\in\Den_\x(S):\textstyle\int_S\al=\int_S\mu\bigr\}
	\end{equation}
	is connected. 
	Correspondingly, 
	\[
		\Gr^{\wt}_{S,\mu}(M)=\bigl\{(N,\nu)\in\Gr^{\wt}_S(M):\textstyle\int_N\nu=\int_S\mu\bigr\}.
	\]
	This is the case considered in \cite{W90,Lee,GBV17}. 
	
	If $S$ is built out of two diffeomorphic connected components, then $H^k(S;\O_S)\cong\R^2$ and any diffeomorphism swapping the two connected components acts nontrivially on this cohomology.
	If $\mu$ has equal total volume on the two connected components, then the orbit $H^k(S;\O_S)_{[\mu]}$ is a one-point set and $\Den(S)_\mu$ is connected.
	Otherwise $H^k(S;\O_S)_{[\mu]}$ consists of two points and, by \eqref{E:DenSmu}, $\Den(S)_\mu$ has two connected components.
\end{example}

\begin{remark}\label{R:EmbGr}
	Suppose $\mu\in\Den_\x(S)$.
	It is well known that $\Diff(S,\mu)$, the group of diffeomorphisms preserving $\mu$, is a splitting Lie subgroup in $\Diff(S)$, see \cite[Theorem~III.2.5.3 on page~203]{Ham}.
	Moreover, the map provided by the action, $\Diff(S)\to\Den(S)_\mu$, $f\mapsto f_*\mu$, is a smooth principal bundle with structure group $\Diff(S,\mu)$.
	Via \eqref{E:grswmu} this implies that the surjective and $\Diff(M)$ equivariant map 
	\[
		\Emb_S(M)\to\Gr_{S,\mu}^{\wt}(M),\qquad\varphi\mapsto\bigl(\varphi(S),\varphi_*\mu\bigr),
	\]
	is smooth principal bundle with structure group $\Diff(S,\mu)$.
\end{remark}

\begin{remark}\label{R:DiffGr}
	Suppose $(N,\nu)\in\Gr_{S}^{\wt}(M)$ and let $\Gr_{S}^{\wt}(M)_{(N,\nu)}$ denote its $\Diff_c(M)$ orbit.
	Combining the preceeding remark with the fact that the $\Diff_c(M)$ action on $\Emb_S(M)$ admits local smooth sections \cite[Lemma~2.1(c)]{HV20}, we see that the map provided by the action, \[\Diff_c(M)\to\Gr_{S}^{\wt}(M)_{(N,\nu)},\qquad f\mapsto\bigl(f(N),f_*\nu\bigr)\] is a smooth principal bundle with structure group $\Diff_c(M,N,\nu)$, the group of diffeomorphisms preserving $N$ and $\nu$.
	The latter is a splitting Lie subgroup in $\Diff_c(M)$, for it coincides with the preimage of $\Diff(N,\nu)$ under the canonical bundle projection $\Diff_c(M,N)\to\Diff(N)$, see \cite[Lemma~2.1(d)]{HV20}.
	Hence, each orbit may be regarded as a homogeneous space,
	\[
		\Gr_S^{\wt}(M)_{(N,\nu)}=\Diff_c(M)/\Diff_c(M,N,\nu).
	\]
\end{remark}

\subsection{Weighted nonlinear flag manifolds}\label{pre}

Fix natural numbers $k_i$ such that 
\begin{equation}\label{E:dimSi}
	0\leq k_1<k_2<\cdots<k_r
\end{equation}
and let $\S=(S_1,\dotsc,S_r)$ be a collection of closed smooth manifolds with $\dim S_i=k_i$.

For a smooth manifold $M$ we let
$$
	\Flag_{\S}(M)
	:=\left\{\bigl(N_1,\dots,N_r)\in\prod_{i=1}^r\Gr_{S_i}(M)\middle|\forall i:N_i\subseteq N_{i+1}\right\}
$$
denote the space of \emph{nonlinear flags of type $\S$ in $M$,} and we write
$$
	\pFlag_{\S}(M)
	:=\left\{(\varphi_1,\dotsc,\varphi_r)\in\prod_{i=1}^r\Emb_{S_i}(M)\middle|\forall i:\varphi_i(S_i)\subseteq\varphi_{i+1}(S_{i+1})\right\}
$$
for the space of the space of \emph{nonlinear frames of type $\S$ in $M$.}
In \cite[Proposition~2.3]{HV20} it has been shown that $\Flag_{\S}(M)$ and $\pFlag_{\S}(M)$ are splitting smooth submanifolds of $\prod_{i=1}^r\Gr_{S_i}(M)$ and $\prod_{i=1}^r\Emb_{S_i}(M)$, respectively.
Moreover, the canonical $\Diff(M)$ equivariant map 
\begin{equation}\label{1}
	\pFlag_{\mathcal S}(M)\to\Flag_{\mathcal S}(M),
	\qquad(\ph_1,\dots,\ph_r)\mapsto\bigl(\ph_1(S_1),\dotsc,\ph_r(S_r)\bigr)
\end{equation} 
is a smooth principal fiber bundle with structure group 
$$
	\Diff(\mathcal S):=\prod_{i=1}^r\Diff(S_i).
$$

We denote the space of \emph{weighted nonlinear flags of type $\S$ in $M$} by
\begin{equation}\label{1717}
	\Flag^{\wt}_{\S}(M)
	:=\left\{\bigl((N_1,\nu_1),\dotsc,(N_r,\nu_r)\bigr)\in\prod_{i=1}^r\Gr^{\wt}_{S_i}(M)\middle|\forall i:N_i\subseteq N_{i+1}\right\}.
\end{equation}
This is a splitting smooth submanifold in $\prod_{i=1}^r\Gr^{\wt}_{S_i}(M)$, for it coincides with the preimage of the splitting smooth submanifold $\Flag_{\S}(M)$ under the bundle projection $\prod_{i=1}^r\Gr^{\wt}_{S_i}(M)\to\prod_{i=1}^r\Gr_{S_i}(M)$.
Moreover, the canonical $\Diff(M)$ equivariant forgetful map
\begin{equation}\label{E:Flagwtp}
	\Flag_{\mathcal S}^{\wt}(M)\to\Flag_{\mathcal S}(M),\quad\bigl((N_1,\nu_1),\dotsc,(N_r,\nu_r)\bigr)\mapsto(N_1,\dotsc,N_r),
\end{equation}
is a smooth fiber bundle with typical fiber
\begin{equation}\label{not}
	\Den_\x(\mathcal S):=\prod_{i=1}^r\Den_\x(S_i).
\end{equation}
The latter is a $\Diff(\S)$ invariant open subset in the Fr\'echet space $\Den(\mathcal S):=\prod_{i=1}^r\Den(S_i)$.
Furthermore, the canonical $\Diff(M)$ equivariant bijection
\begin{align}\label{asoc}
	\Flag_{\mathcal S}^{\wt}(M)&=\pFlag_{\mathcal S}(M)\times_{\Diff(\mathcal S)}\Den_\x(\mathcal S),\\\notag
	\Bigl(\bigl(\varphi_1(S_1),(\varphi_1)_*\mu_1\bigr),\dotsc,\bigl(\varphi_r(S_r),(\varphi_r)_*\mu_r\bigr)\Bigr)&\leftrightarrow\bigl[(\varphi_1,\dotsc,\varphi_r),(\mu_1,\dotsc,\mu_r)\bigr],
\end{align}
is a diffeomorphism between $\Flag_{\mathcal S}^{\wt}(M)$ and the bundle associated to the nonlinear frame bundle in \eqref{1} and the canonical $\Diff(\S)$ action on $\Den_\x(\S)$.
Indeed, this is just the bundle diffeomorphism $\prod_{i=1}^r\Gr^{\wt}_{S_i}(M)=\prod_{i=1}^r\Emb_{S_i}(M)\times_{\Diff(S_i)}\Den_\times(S_i)$ obtained by taking the product of the diffeomorphisms in \eqref{grsw}, restricted over the submanifold $\Flag_{\S}(M)$ in its base $\prod_{i=1}^r\Gr_{S_i}(M)$.

We have a canonical $\Diff(M)$ equivariant map
\begin{equation}\label{pl}
	J:\Flag^{\wt}_{\mathcal S}(M)\to C^\infty(M)^*,\quad
	\bigl\langle J((N_1,\nu_1),\dotsc,(N_r,\nu_r)),f\bigr\rangle:=\sum_{i=1}^r\int_{N_i}f\nu_i.
\end{equation}
The image of $J$ consists of compactly supported distributions on $M$ with mild singularities.
More precisely, the wave front set of $J((N_1,\nu_1),\dotsc,(N_r,\nu_r))$ coincides with the union of the conormal bundles of $N_1,\dotsc,N_r$.

\begin{lemma}\label{injectiv}
	The map in \eqref{pl} is injective.
\end{lemma}

\begin{proof}
	Suppose $J((N_1,\nu_1),\dotsc,(N_r,\nu_r))=J((N_1',\nu_1'),\dotsc,(N_r',\nu_r'))$.
	Proceeding by induction on $r$, it suffices to show $N_r=N_r'$ and $\nu_r=\nu_r'$.
	
	To show $N_r=N_r'$, we assume by contradiction that there exists $x\in N_r$ with $x\notin N_r'$.
	Using \eqref{E:dimSi}, we see that $N_r\setminus N_{r-1}$ is dense in $N_r$.
	Thus, we may w.l.o.g.\ assume $x\notin N_{r-1}$.
	Moreover, $\nu_r(x)\neq0$ as $\nu_r$ does not vanish on $N_r$.
	Hence, if $f$ is a smooth bump function supported on a sufficiently small neighborhood of $x$, then $\langle J((N_1,\nu_1),\dotsc,(N_r,\nu_r)),f\rangle=\int_{N_r}f\nu_r\neq0$ and $\langle J((N_1',\nu_1'),\dotsc,(N_r',\nu_r')),f\rangle=0$.
	Since this contradicts our assumption, we conclude $N_r=N_r'$.
	
	To show $\nu_r=\nu_r'$, we assume by contradiction that there exists $x\in N_r=N_r'$ with $\nu_r(x)\neq\nu_r'(x)$.
	As before, we may w.l.o.g.\ assume $x\notin N_{r-1}$ and $x\notin N_{r-1}'$.
	Hence, if $f$ is a smooth bump function supported in a sufficiently small neighborhood of $x$, then $\langle J((N_1,\nu_1),\dotsc,(N_r,\nu_r)),f\rangle=\int_{N_r}f\nu_r\neq\int_{N_r'}f\nu_r'=\langle J((N_1',\nu_1'),\dotsc,(N_r',\nu_r')),f\rangle$.
	Since this contradicts our assumption, we conclude $\nu_r=\nu_r'$.
\end{proof}

\begin{remark}
	Suppose $\omega$ is a symplectic form on $M$, and let $\Flag_{\S}^{\symp}(M)$ denote {the} manifold of symplectic nonlinear flags of type $\S$, cf.~\cite[Section~4.2]{HV20}.
	Recall that this is the open subset consisting of all flags $(N_1,\dotsc,N_r)\in\Flag_{\S}(M)$ such that $\omega$ restricts to a symplectic form on each $N_i$.
	Hence, $k_i$ must be even and $\omega^{k_i/2}$ pulls back to a volume form on $N_i$ which in turn gives rise to a volume density $\nu_i=|\iota_{N_i}^*\omega^{k_i/2}|$ on $N_i$.
	Consequently, the symplectic form $\omega$ provides a $\Symp(M,\omega)$ equivariant injective smooth map (section) 
	\begin{equation}\label{E:symp.sec}
		\Flag^{\symp}_{\S}(M)\to\Flag_{\S}^{\wt}(M)
	\end{equation}
	which is right inverse to the restriction of the canonical bundle projection in \eqref{E:Flagwtp}.
	Composing the map in \eqref{E:symp.sec} with $J$ in \eqref{pl}, we obtain the moment map considered in \cite[Eq.~(38)]{HV20}.
\end{remark}

\begin{remark}
	A Riemannian metric $g$ on $M$ induces a volume density on every submanifold of $M$.
	Hence, $g$ provides a smooth section $\Flag_{\S}(M)\to\Flag_{\S}^{\wt}(M)$ of the canonical bundle projection in \eqref{E:Flagwtp}, which is $\Isom(M,g)$ equivariant, cf.~\cite[Section~6]{W90}.
\end{remark}

\subsection{Reduction of structure group}\label{SS:red}

It will be convenient to use a reduction of the structure group for the principal frame bundle in \eqref{1}.
To this end, we fix embeddings $\iota_i\colon S_i\to S_{i+1}$ and put $\iota=(\iota_1,\dotsc,\iota_{r-1})$.

We begin by recalling some facts from \cite[Proposition~2.10]{HV20}.
The space of \emph{nonlinear flags of type $(\S,\iota)$ in $M$,}
\[
	\Flag_{\S,\iota}(M):=\left\{(N_1,\dotsc,N_r)\in\Flag_{\S}(M)\middle|\begin{array}{c}\bigl(S_1\xrightarrow{\iota_1}S_2\xrightarrow{\iota_2}\cdots\to S_r\bigr)\\
	\cong\bigl(N_1\subseteq N_2\subseteq\cdots\subseteq N_r\bigr)\end{array}\right\},
\]
consists of all nonlinear flags $(N_1,\dotsc,N_r)$ in $M$ such that there exist diffeomorphisms $S_i\to N_i$, $1\leq i\leq r$ intertwining $\iota_i$ with the canonical inclusion $N_i\subseteq N_{i+1}$.
This is a $\Diff(M)$ invariant open and closed subset in $\Flag_{\S}(M)$.
The space of \emph{parametrized nonlinear flags ({nonlinear} frames) of type $(\S,\iota)$ in $M$,}
\[
	\pFlag_{\S,\iota}(M):=\left\{(\varphi_1,\dotsc,\varphi_r)\in\pFlag_{\S}(M)\middle|\forall i:\varphi_i=\varphi_{i+1}\circ\iota_i\right\},
\]
is a splitting smooth submanifold of $\pFlag_{\S}(M)$.
Moreover, the map $\pFlag_{\S,\iota}(M)\to\Flag_{\S,\iota}(M)$ obtained by restriction of \eqref{1} is a smooth principal bundle with structure group 
\begin{equation}\label{E:diff}
	\Diff(\mathcal S;\iota):=\left\{(g_1,\dotsc,g_r)\in\prod_{i=1}^r\Diff(S_i)\middle|\forall i:g_{i+1}\circ\iota_i=\iota_i\circ g_i\right\}.
\end{equation}
The latter is a splitting Lie subgroup in $\Diff(\S)$ with Lie algebra
\begin{equation}\label{88}
	\mathfrak X(\mathcal S;\iota)=\left\{(Z_1,\dotsc,Z_r)\in\prod_{i=1}^r\mathfrak X(S_i)\middle|\forall i:Z_{i+1}\circ\iota_i=T\iota_i\circ Z_i\right\}.
\end{equation}
We obtain a $\Diff(M)$ equivariant commutative diagram
\begin{equation}\label{E:red.diag}
	\vcenter{
	\xymatrix{
		\pFlag_{\S,\iota}(M)\ar[d]_{\Diff(\mathcal S;\iota)}\ar@{^(->}[r]&\pFlag_{\mathcal S}(M)\ar[d]^{\Diff(\mathcal S)}
		\\
		\Flag_{\mathcal S,\iota}(M)\ar@{^(->}[r]&\Flag_{\mathcal S}(M).
	}}
\end{equation}
which may be regarded as a reduction of the structure group for \eqref{1} along the inclusion $\Diff(\mathcal S;\iota)\subseteq\Diff(\S)$ over $\Flag_{\mathcal S,\iota}(M)$, see \cite[Proposition~2.10]{HV20} for more details.

\begin{remark}\label{R:Fr.Emb}
The $\Diff(M)$ equivariant bijection
\begin{equation}\label{E:DiffSiEmb}
	\pFlag_{\S,\iota}(M)=\Emb_{S_r}(M),\qquad(\varphi_1,\dotsc,\varphi_r)\mapsto\varphi_r,
\end{equation}
is a diffeomorphism \cite[Proposition~2.10(b)]{HV20}.
Correspondingly, we have a group isomorphism
\begin{equation}\label{E:DiffSrSigma}
	\Diff(\mathcal S;\iota)=\Diff(S_r;\Sigma),\quad(g_1,\dotsc,g_r)\mapsto g_r,
\end{equation}
	where $\Diff(S_r;\Sigma)$ denotes the subgroup of all diffeomorphisms of $S_r$ preserving the nonlinear flag $\Sigma=(\Sigma_1,\dotsc,\Sigma_{r-1})$ in $S_r$, where $\Sigma_i:=(\iota_{r-1}\circ\cdots\circ\iota_i)(S_i)$.
The latter is a splitting Lie subgroup of $\Diff(S_r)$, see \cite[Proposition~2.9(b)]{HV20}, and \eqref{E:DiffSrSigma} is a diffeomorphism of Lie groups \cite[Proposition~2.10(a)]{HV20}.
The Lie algebra of $\Diff(\S;\io)$, can be identified in a similar way with $\X(S_r;\Si)$, the Lie algebra of vector fields on $S_r$ that are tangent to $\Sigma_1,\dotsc,\Sigma_{r-1}$.
\end{remark}

We are interested in the reduction of structure group \eqref{E:red.diag} because the $\Diff_c(M)$ action on $\pFlag_{\S,\iota}(M)$ admits local smooth sections.
This follows from \cite[Lemma~2.1(c)]{HV20} and the diffeomorphism in \eqref{E:DiffSiEmb}.

Let $\Flag^{\wt}_{\S,\iota}(M)$ denote the preimage of $\Flag_{\S,\iota}(M)$ under the bundle projection in \eqref{E:Flagwtp}.
Restricting the diffeomorphism in \eqref{asoc} over $\Flag_{\S,\iota}(M)$ and combining this with the diffeomorphism in \eqref{E:DiffSiEmb}, we obtain a $\Diff(M)$ equivariant diffeomorphism of bundles over $\Flag_{\S,\iota}(M)$,
\begin{equation}\label{E:FlagwtSiota}
	\Flag^{\wt}_{\S,\iota}(M)=\Emb_{S_r}(M)\times_{\Diff(\S,\iota)}\Den_\x(\S).
%	\Flag^{\wt}_{\S,\iota}(M)=\pFlag_{\S,\iota}(M)\times_{\Diff(\S,\iota)}\Den_\x(\S).
\end{equation}

\subsection{The $\Diff(M)$ action on the space of weighted nonlinear flags}\label{SS:DiffM.Flag}

In this section we aim at describing the $\Diff(M)$ orbits in $\Flag_{\S}^{\wt}(M)$, see Theorem~\ref{T:wtFlag.homog} below.

Let $\iota=(\iota_1,\dotsc,\iota_{r-1})$ be a collection of embeddings $\iota_i:S_i\to S_{i+1}$ and suppose $\mu=(\mu_1,\dotsc,\mu_r)\in\Den_\x(\S)$.
We define the space of \emph{ weighted flags of type $(\S,\iota,\mu)$ in $M$} by
\[
	\Flag^{\wt}_{\S,\iota,\mu}(M):=\left\{\bigl((N_1,\nu_1),\dotsc,(N_r,\nu_r)\bigr)\in\Flag^{\wt}_{\S}(M)\middle|\begin{array}{c}\bigl(S_1\xrightarrow{\iota_1}S_2\xrightarrow{\iota_2}\cdots\to S_r,\mu_1,\dotsc,\mu_r\bigr)\\\cong\bigl(N_1\subseteq N_2\subseteq\cdots\subseteq N_r,\nu_1,\dotsc,\nu_r\bigr)\end{array}\right\},
\]
that is, the space of all weighted flags $\bigl((N_1,\nu_1),\dotsc,(N_r,\nu_r)\bigr)$ in $M$ such that there exist diffeomorphisms $S_i\to N_i$, $1\leq i\leq r$, intertwining $\iota_i$ with the canonical inclusion $N_i\subseteq N_{i+1}$, and taking $\mu_i$ to $\nu_i$.

Denoting the $\Diff(\S,\iota)$ orbit of $\mu$ by $\Den(\S)_{\iota,\mu}$, the diffeomorphism in \eqref{E:FlagwtSiota} restricts to a $\Diff(M)$ equivariant bijection
\begin{equation}\label{Simu}
%	\Flag^{\wt}_{\S,\iota,\mu}(M)=\pFlag_{\S,\iota}(M)\times_{\Diff(\S,\iota)}\Den(\S)_{\iota,\mu}.
	\Flag^{\wt}_{\S,\iota,\mu}(M)=\Emb_{S_r}(M)\times_{\Diff(\S,\iota)}\Den(\S)_{\iota,\mu}.
\end{equation}

Consider the finite dimensional vector space
\begin{equation}\label{HS}
	H(\mathcal S,\iota)
	:=\prod_{i=1}^rH^{k_i}\bigl(S_i,\iota_{i-1}(S_{i-1});\mathcal O_{S_i}\bigr)
	=\prod_{i=1}^rH_0\bigl(S_i\setminus\iota_{i-1}(S_{i-1});\mathbb R\bigr).
\end{equation}
Here the left hand side denotes relative de~Rham cohomology with coefficients in the orientation bundle, and we are using the convention $S_0=\emptyset$.
The $\Diff(\S,\iota)$ equivariant identification on the right hand side indicates Poincar\'e--Lefschetz duality.
We have a $\Diff(\S,\iota)$ equivariant linear map
\begin{equation}\label{E:hiota}
	h_{\S,\iota}\colon\Den(\mathcal S)\to H(\mathcal S,\iota),\qquad h_{\S,\iota}(\mu_1,\dotsc,\mu_r):=\bigl([\mu_1],\dotsc,[\mu_r]\bigr).
\end{equation}
Pinning down the class $[\mu]:=h_{\S,\iota}(\mu)$ thus amounts to specifying the integrals of $\mu_i$ over each connected component of $S_i\setminus\iota_{i-1}(S_{i-1})$ for $i=1,\dotsc,r$.

\begin{remark}[Large codimensions]\label{larg}
	If the codimensions $\dim(S_i)-\dim(S_{i-1})$ are all strictly larger than one, then $H^{k_i}\bigl(S_i,\iota_{i-1}(S_{i-1});\O_{S_i}\bigr)=H^{k_i}\bigl(S_i;\O_{S_i}\bigr)=H_0(S_i;\R)$, and
	\[
		H(\S,\io)=\prod_{i=1}^{r}H^{k_i}\bigl(S_i;\O_{S_i}\bigr)
		=\prod_{i=1}^{r}H_0\bigl(S_i;\RR\bigr).
	\]
	Hence, in this case, the cohomology space $H(\S,\io)$ does not depend on the embeddings $\io$.
\end{remark}

\begin{proposition}\label{P:DenS.homog}
	In this situation the following hold true:
	\begin{enumerate}[(a)]
		\item\label{I:DenS.homog.orbito}
			The $\Diff(\mathcal S,\iota)_0$ orbit of $\mu$ coincides with the convex set $\Den_\times(\mathcal S)\cap h_{\S,\iota}^{-1}([\mu])$.
			In particular, this orbit is a splitting smooth submanifold in $\Den_\x(\mathcal S)$ with finite codimension $\dim H(\S,\iota)$.
		\item\label{I:DenS.homog.seco}
			The $\Diff(\mathcal S,\iota)_0$ action on $\Den_\times(\mathcal S)\cap h_{\S,\iota}^{-1}([\mu])$ admits local smooth sections.
		\item\label{I:DenS.homog.orbit}
			Denoting the (finite) $\Diff(\mathcal S;\iota)$ orbit of $[\mu]$ by $H(\mathcal S,\iota)_{[\mu]}$, 
			the $\Diff(\S;\io)$ orbit of $\mu$ is
			\begin{equation}\label{densim}
				\Den(\mathcal S)_{\iota,\mu}=\Den_\times(\mathcal S)\cap h_{\S,\iota}^{-1}\bigl(H(\mathcal S,\iota)_{[\mu]}\bigr).
			\end{equation}
			In particular, $\Den(\mathcal S)_{\iota,\mu}$ is a splitting smooth submanifold in $\Den_\x(\mathcal S)$ with finite codimension $\dim H(\S,\iota)$ and with tangent spaces 
			\begin{equation}\label{tans}
				T_\al\Den(\mathcal S)_{\iota,\mu}
				=\ker h_{\S,\iota}
				=\bigl\{(d\ga_1,\dotsc,d\ga_r):\ga_i\in\Om^{k_i-1}(S_i;\mathcal O_{S_i}),\iota_{i-1}^*\ga_i=0\bigr\}.
			\end{equation}
			Moreover, the $\Diff(\mathcal S,\iota)$ action on $\Den(\mathcal S)_{\io,\mu}$ admits local smooth sections.
		\item\label{I:DenS.homog.Gr}
			The canonical inclusion $\Den(\mathcal S)_{\iota,\mu}\subseteq\prod_{i=1}^r\Den(S_i)_{\mu_i}$ is a splitting smooth submanifold of finite codimension.
	\end{enumerate}
\end{proposition}

We postpone the proof of this proposition and proceed with the main result in this section:

\begin{theorem}\label{T:wtFlag.homog}
	In this situation the following hold true:
	\begin{enumerate}[(a)]
		\item	The space $\Flag^{\wt}_{\S,\iota,\mu}(M)$ is a splitting smooth submanifold in $\Flag^{\wt}_{\S,\iota}(M)$ with finite codimension $\dim H(\S,\iota)$.
		\item	The canonical $\Diff(M)$ equivariant forgetful map $\Flag^{\wt}_{\S,\iota,\mu}(M)\to\Flag_{\S,\iota}(M)$ is a locally trivial smooth fiber bundle with typical fiber $\Den(\S)_{\iota,\mu}$.
		\item	The canonical inclusion $\Flag^{\wt}_{\S,\iota,\mu}(M)\subseteq\prod_{i=1}^r\Gr_{S_i,\mu_i}^{\wt}(M)$ is a splitting smooth submanifold.
		\item	The $\Diff_c(M)$ action on $\Flag^{\wt}_{\S,\iota,\mu}(M)$ admits local smooth sections.
			In particular, each connected component of $\Flag^{\wt}_{\S,\iota,\mu}(M)$ is a $\Diff_c(M)_0$ orbit.
			Furthermore, every $\Diff(M)$ or $\Diff_c(M)$ orbit in $\Flag^{\wt}_{\S,\iota,\mu}(M)$ is a union of connected components.
	\end{enumerate}
\end{theorem}

\begin{proof}
	Parts (a) and (b) follow by combining \eqref{E:FlagwtSiota} and \eqref{Simu} with Proposition~\ref{P:DenS.homog}\eqref{I:DenS.homog.orbit}.

	Part (c) follows from Proposition~\ref{P:DenS.homog}\eqref{I:DenS.homog.Gr} and the reduction of structure groups in \eqref{E:red.diag} via the diffeomorphisms in \eqref{E:grswmu} and \eqref{Simu}, see also \eqref{E:DiffSiEmb}.

	Let us finally turn to part (d).
	By Proposition~\ref{P:DenS.homog}\eqref{I:DenS.homog.orbit}, the (transitive) $\Diff(\mathcal S;\iota)$ action on $\Den(\mathcal S)_{\iota,\mu}$ admits local smooth sections.
%	Combining \cite[Proposition~2.10(b)]{HV20} with \cite[Lemma~2.1(c)]{HV20}, we see that the $\Diff_c(M)$ action on $\pFlag_{\mathcal S,\iota}(M)=\Emb_{S_r}(M)$ admits local smooth sections.
	The $\Diff_c(M)$ action on $\Emb_{S_r}(M)$ admits local smooth sections too, cf.~\cite[Lemma~2.1(c)]{HV20}.
	Using Lemma~\ref{before.sec}, we conclude that the $\Diff_c(M)$ action on $\Flag^{\wt}_{\S,\iota,\mu}(M)$ admits local smooth sections, cf.~\eqref{E:FlagwtSiota}.
\end{proof}

We will prove Proposition~\ref{P:DenS.homog} by induction on the depth of the flags, using the following crucial lemma whose proof we postpone.

\begin{lemma}\label{L:diffmuS}
	Let $S$ be a closed submanifold of $N$ such that $\dim(S)<\dim(N)=n$, and consider the $\Diff(N,S)$ equivariant linear map
	\[
		h:\Den(N)\to H^n(N,S;\mathcal O_N),\qquad h(\mu)=[\mu].
	\]
	Then, for each $\kappa\in H^n(N,S;\mathcal O_N)$, the natural $\Diff(N,S)_0$ action on
	\begin{equation}\label{E:abc}
        	\Diff(S)\times\bigl(\Den_\times(N)\cap h^{-1}(\kappa)\bigr)
	\end{equation}
	admits local smooth sections.
\end{lemma}

\begin{proof}[{Proof of Proposition~\ref{P:DenS.homog}}]
	We proceed by induction on $r$ using Lemma~\ref{L:diffmuS}.
	Let us denote the truncated sequences by $\mathcal S':=(S_1,\dotsc,S_{r-1})$, $\mu':=(\mu_1,\dotsc,\mu_{r-1})$, and $\iota':=(\iota_1,\dotsc,\iota_{r-2})$.
	By induction, the $\Diff(\mathcal S',\iota')_0$ action on $\Den_\times(\mathcal S')\cap h_{\S',\iota'}^{-1}([\mu'])$ admits local smooth sections.
	Hence, there exist an open neighborhood $U'$ of $\mu'$ in $\Den_\times(\mathcal S')\cap h_{\S',\iota'}^{-1}([\mu'])$ and a smooth map
	\[
		\Den_\times(\mathcal S')\cap h_{\S',\iota'}^{-1}([\mu'])\supseteq U'\xrightarrow{\sigma'}\Diff(\mathcal S';\iota')_0,
	\]
	such that for all $\tilde\mu'\in U'$ we have
	\begin{equation}
		\bigl(\sigma'(\tilde\mu')\bigr)_*\mu'=\tilde\mu'\qquad\text{and}\qquad\sigma'(\mu')=\id.
	\end{equation}

	Recall that $\Diff(\mathcal S',\iota')$ is a splitting Lie subgroup in $\Diff(S_{r-1})$, cf.~\cite[Propositions~2.9(b) and 2.10(a)]{HV20}.
	Using \cite[Lemma~2.1(d)]{HV20} this implies that $\Diff(\mathcal S,\iota)$ is a splitting Lie subgroup in $\Diff(S_r,\iota_{r-1}(S_{r-1}))$.
	Hence, restricting a local smooth section as in Lemma~\ref{L:diffmuS}, we see that the $\Diff(\mathcal S,\iota)_0$ action on $\Diff(\mathcal S',\iota')\times\left(\Den_\times(S_r)\cap h^{-1}([\mu_r])\right)$ admits local smooth sections.
	In other words, there exist an open neighborhood $V$ of the identity in $\Diff(\mathcal S';\iota')$, an open neighborhood $U''$ of $\mu_r$ in $\Den_\times(S_r)\cap h^{-1}([\mu_r])$, and a smooth map
	\[
		\Diff(\mathcal S';\iota')\times\left(\Den_\times(S_r)\cap h^{-1}([\mu_r])\right)\supseteq V\times U''\xrightarrow{\sigma''}\Diff(\mathcal S;\iota)_0,
	\]
	such that for all $g\in V$ and $\tilde\mu_r\in U''$ we have
	\begin{equation}
		\sigma''(g,\tilde\mu_r)\cdot(\id,\mu_r)=(g,\tilde\mu_r)\qquad\text{and}\qquad\sigma''(\id,\mu_r)=\id.
	\end{equation}

	Hence, $U:=(\sigma')^{-1}(V)\times U''$ is an open neighborhood of $\mu$ in $\Den_\times(\mathcal S)\cap h_{\S,\iota}^{-1}([\mu])$, and
	\[
		\Den_\times(\mathcal S)\cap h_{\S,\iota}^{-1}([\mu])\supseteq U\xrightarrow\sigma\Diff(\mathcal S;\iota)_0,
	        \qquad\sigma(\tilde\mu_1,\dots,\tilde\mu_r):=\sigma''(\sigma'(\tilde\mu_1,\dotsc,\tilde\mu_{r-1}),\tilde\mu_r)
	\]
	is a local smooth section for the $\Diff(\mathcal S;\iota)_0$ action on $\Den_\times(\mathcal S)\cap h_{\S,\iota}^{-1}([\mu])$, i.e.,
	\[
        	\sigma(\mu)=\id\qquad\text{and}\qquad\sigma(\tilde\mu)_*\mu=\tilde\mu,
	\]
	for all $\tilde\mu\in U$.
	By convexity, $\Den_\times(\mathcal S)\cap h_{\S,\iota}^{-1}([\mu])$ is connected.
	Therefore, the $\Diff(\mathcal S;\iota)_0$ action is transitive on $\Den_\times(\mathcal S)\cap h_{\S,\iota}^{-1}([\mu])$.
	This shows \eqref{I:DenS.homog.orbito} and \eqref{I:DenS.homog.seco}.
	Part \eqref{I:DenS.homog.orbit} follows immediately.

	To see \eqref{I:DenS.homog.Gr}, let $A$ denote the preimage of $\prod_{i=1}^rH^{k_i}(S_i;\mathcal O_{S_i})_{[\mu_i]}$ under the canonical linear surjection $H(\S,\io)\to\prod_{i=1}^rH^{k_i}(S_i;\mathcal O_{S_i})$.
	Hence, $A$ is a finite union of affine subspaces in $H(\S,\io)$.
	In view of \eqref{E:DenSmu} we have $\prod_{i=1}^r\Den(S_i)_{\mu_i}=\Den_\x(\S)\cap h_{\S,\iota}^{-1}(A)$.
	Combining this with \eqref{densim}, we conclude that $\Den(\S)_{\iota,\mu}$ is a splitting smooth submanifold in $\prod_{i=1}^r\Den(S_i)_{\mu_i}$ with finite codimension $\dim A$.
\end{proof}

Let us next establish the following infinitesimal version of Lemma~\ref{L:diffmuS}:

\begin{lemma}\label{L:irM}
	Let $S$ be a closed submanifold of $N$ such that $\dim(S)<\dim(N)=n$.
	Suppose $\mu\in\Den_\x(N)$, $\gamma\in\Omega^{n-1}(N,S;\O_N):=\{\al\in\Om(N;\O_N):\io_S^*\al=0\}$, and $Z\in\mathfrak X(S)$.
	Then there exists a vector field $X\in\mathfrak X(N)$ such that $L_X\mu=d\gamma$ and $X|_S=Z$.
\end{lemma}

\begin{proof}
	Let $\tilde Z\in\mathfrak X(N)$ be any extension of $Z$, i.e.\ $\tilde Z|_S=Z$. Note that $i_{\tilde Z}\mu\in\Om^{n-1}(N;\O_N)$ vanishes when pulled back to $S$, hence the same holds for $\be:=\gamma-i_{\tilde Z}\mu\in\Omega^{n-1}(N;\O_N)$.

	Let us first construct $\tilde\gamma\in\Omega^{n-1}(N;\O_N)$ such that $d\tilde\gamma=d\be$ and $\tilde\gamma|_S=0$.
	To this end, we fix a smooth homotopy $h\colon N\times[0,1]\to N$ such that $h_1=\id_N$, $h_t|_S=\id_S$, and such that $h_0$ maps a neighborhood of $S$ into $S$.
	Consider the corresponding chain homotopy $\phi\colon\Omega^*(N;\O_N)\to\Omega^{*-1}(N;\O_N)$ defined by $\phi(\alpha):=\int_0^1\iota_t^*i_{\partial_t}h^*\alpha\,dt$, where $\iota_t\colon N\to N\times[0,1]$ denotes the inclusion at $t$, that is, $\iota_t(x):=(x,t)$.
	Then $\phi(\alpha)|_S=0$ and $d(\phi(\alpha))+\phi(d\alpha)=h_1^*\alpha-h_0^*\alpha$, for all forms $\alpha\in\Omega^*(N;\O_N)$.
	In particular, $d\alpha=d\bigl(h_0^*\alpha+\phi(d\alpha)\bigr)$.
	Defining $\tilde\gamma:=h_0^*\be+\phi(d\be)$, we obtain $d\tilde\gamma=d\be$. Moreover, $h_0^*\be|_S=0$ because $\be$ vanishes when pulled back to $S$, hence $\tilde\gamma|_S=0$ as desired.
	
	Defining a vector field $Y\in\mathfrak X(N)$ by $i_Y\mu:=\tilde\gamma$, we obtain $Y|_S=0$ and $L_Y\mu=di_Y\mu=d\tilde\gamma=d(\gamma-i_{\tilde Z}\mu)=d\gamma-L_{\tilde Z}\mu$.
	Hence, the vector field $X:=\tilde Z+Y$ has the desired properties.
\end{proof}

\begin{proof}[{Proof of Lemma~\ref{L:diffmuS}}]
	Recall that the infinitesimal $\Diff(N,S)$ action on $\Diff(S)\times\Den(N)$ is
	\[
		\zeta_X(f,\mu)=\bigl(R_{X|_S}(f),-L_X\mu\bigr),
	\]
	where $X\in\mathfrak X(N,S)$, $f\in\Diff(S)$, and $\mu\in\Den(S)$.
	Here, for $Z\in T_{\id}\Diff(S)=\mathfrak X(S)$ we let $R_Z(f)$ denote the right invariant vector field on $\Diff(S)$ such that $R_Z(\id)=Z$.

	Note that the vector field $\tilde Z$ in the proof of Lemma~\ref{L:irM} can be chosen to depend smoothly (and linearly) on $Z$.
	Hence, the proof of said lemma actually provides a smooth map
	\[
		\tilde\sigma:T_{\id}\Diff(S)\times T\bigl(\Den_\times(N)\cap h^{-1}(\kappa)\bigr)\to\mathfrak X(N,S)
	\]
	such that $\zeta_{\tilde\sigma(Z,d\gamma)}(\id,\mu)=(Z,d\gamma)$ for all $Z\in\mathfrak X(S)=T_{\id}\Diff(S)$, $\mu\in\Den_\times(N)\cap h^{-1}(\kappa)$, and $d\gamma\in d\Omega^{n-1}(N,S;\mathcal O_N)=T_\mu\bigl(\Den_\times(N)\cap h^{-1}(\kappa)\bigr)$. 
	Combining this with the right trivialization of $T\Diff(S)$, we obtain a smooth map 
	\[
		\sigma:T\Bigl(\Diff(S)\times\bigl(\Den_\times(N)\cap h^{-1}(\kappa)\bigr)\Bigr)\to\mathfrak X(N,S)
	\]
	such that $\zeta_{\sigma(Z,\xi)}(f,\mu)=(Z,\xi)$ for all $f\in\Diff(S)$, $Z\in T_f\Diff(S)$, $\mu\in\Den_\times(N)\cap h^{-1}(\kappa)$, and $\xi\in T_\mu\bigl(\Den_\times(N)\cap h^{-1}(\kappa)\bigr)$. 
	As $\Diff(N,S)$ is a regular Lie group, we may apply Lemma~\ref{L:loc.sec} to conclude that the $\Diff(N,S)_0$ action on \eqref{E:abc} admits local smooth sections.
\end{proof}

This completes the proof of Theorem~\ref{T:wtFlag.homog}.

\begin{remark}
%	Fixing volume densities $\mu=(\mu_1,\dotsc,\mu_r)\in\Den_\x(\S)$ allows us to view manifolds of weighted nonlinear flags as base manifolds of principal bundles of nonlinear frames.
	In view of Remark~\ref{R:EmbGr}, we expect that the isotropy subgroup 
	\begin{equation}\label{99}
		\Diff(\mathcal S;\iota,\mu):=\bigl\{(g_1,\dotsc,g_r)\in\Diff(\S,\iota)\bigm|\forall i:g_i^*\mu_i=\mu_i\bigr\}
	\end{equation}
	is a splitting Lie subgroup of $\Diff(\S;\iota)$ with Lie algebra
	\begin{equation}\label{8888}
		\mathfrak X(\mathcal S,\iota,\mu)=\bigl\{(Z_1,\dotsc,Z_r)\in\mathfrak X(\S,\iota)\bigm|\forall i:L_{Z_i}\mu_i=0\bigr\},
	\end{equation}
	and the surjective map provided by the action, $\Diff(\S,\iota)\to\Den(\S)_{\iota,\mu}$, is a locally trivial smooth principal fiber bundle with structure group $\Diff(\S;\io,\mu)$.
	This would follow in a rather {straigthforward} manner, via induction on the depth of the flags, if one could show that the isotropy group $\{f\in\Diff(N,S):f|_S=\id,f^*\mu=\mu\}$ is a splitting Lie subgroup in $\Diff(N,S)$, whenever $S$ is a closed submanifold of $N$ and $\mu$ is a volume density on $N$.
	The proof in \cite[Theorem~III.2.5.3 on page~203]{Ham} covers the case $S=\emptyset$.
	However, the adaptation of said proof to nontrivial $S$ is not entirely {straigthforward}, and we will not attempt to prove this here.
	Note that via the diffeomorphism in \eqref{E:DiffSrSigma} the group $\Diff(\mathcal S;\iota,\mu)$ corresponds to the subgroup of $\Diff(S_r;\Sigma)$ consisting of all diffeomorphisms that preserve $\mu_r$ and whose restriction to $\Sigma_i$ preserves $(\iota_{r-1}\circ\cdots\circ\iota_i)_*\mu_i$, for $1\leq i\leq r-1$.
	Similarly, the Lie algebra $\mathfrak X(\mathcal S;\iota,\mu)$ can be identified to the corresponding subalgebra of $\X(S_r;\Si)$.

	If the expectation formulated in the preceding paragraph were indeed true, then the surjective and $\Diff(M)$ equivariant map
	\begin{align}\label{prbu}
		\Emb_{S_r}(M)=\pFlag_{\S,\iota}(M)&\to\Flag^{\wt}_{\mathcal S,\iota,\mu}(M),\\
		(\varphi_1,\dotsc,\varphi_r)&\mapsto\Bigl(\bigl(\varphi_1(S_1),(\varphi_1)_*\mu_1\bigr),\dotsc,\bigl(\varphi_r(S_r),(\varphi_r)_*\mu_r\bigr)\Bigr),
	\end{align} 
	would be a locally trivial smooth principal fiber bundle with structure group $\Diff(\S;\io,\mu)$, see~\eqref{Simu}.
	Moreover, generalizing Remark~\ref{R:DiffGr}, the isotropy group of a weighted flag $(\N,\nu)$,
	\[
		\Diff_c(M;\mathcal N,\nu):=\bigl\{g\in\Diff_c(M)\bigm|\forall i:g(N_i)=N_i,g|_{N_i}^*\nu_i=\nu_i\bigr\},
	\]
	would be a splitting Lie subgroup of $\Diff_c(M)$, for it coincides with the preimage of $\Diff(\S;\iota,\mu)$ under the bundle projection $\Diff_c(M;\mathcal N)\to\Diff(\N,\iota_{\N})$, cf.~\cite[Lemma 2.1(d), Proposition~2.9(b), and Proposition 2.10(a)]{HV20}.
	Furthermore, the orbit map $\Diff_c(M)\to\Flag^{\wt}_{\mathcal S}(M)_{(\mathcal N,\nu)}$ provided by the action would be a locally trivial smooth principal bundle with structure group $\Diff_c(M;\mathcal N,\nu)$.
	Hence, the $\Diff_c(M)$ orbit of $(\N,\nu)$ could be regarded as a homogeneous space
	\[
		\Flag_{\mathcal S}^{\wt}(M)_{(\mathcal N,\nu)}=\Diff_c(M)/\Diff_c(M;\mathcal N,\nu).
	\]
\end{remark}

\subsection{A homological description}\label{26}

In this section we give a more explicit description of the manifold $\Flag^{\wt}_{\S,\io,\mu}(M)$.

If $\mathcal N=(N_1,\dotsc,N_r)$ is a nonlinear flag of type $\S$ in $M$, we put
\[
	H(\mathcal N):=\prod_{i=1}^rH^{k_i}\bigl(N_i,N_{i-1};\mathcal O_{N_i}\bigr)=\prod_{i=1}^rH_0\bigl(N_i\setminus N_{i-1};\RR\bigr),
\]
with the convention that $N_0=\emptyset$.

We define the space of \emph{homologically weighted flags of type $\S$ in $M$} by
\[
	\Flag^{\hwt}_{\S}(M):=\bigl\{(\mathcal N,[\nu]):\mathcal N\in\Flag^{\wt}_{\S}(M),[\nu]\in H(\mathcal N)\bigr\}
\]
Note that we have a $\Diff(M)$ equivariant forgetful map
\begin{equation}\label{E:Flaghwtp}
	\Flag^{\hwt}_{\S}(M)\to\Flag_{\S}(M),
\end{equation}
as well as a $\Diff(M)$ equivariant map
\begin{equation}\label{E:hS}
	\textcolor{black}{h_{\Flag_{\S}(M)}:\Flag^{\wt}_{\S}(M)\to\Flag^{\hwt}_{\S}(M),\qquad (\mathcal N,\nu)\mapsto(\mathcal N,[\nu]).}
\end{equation}

Let $\Flag^{\hwt}_{\S,\iota}(M)$ denote the preimage of the open subset $\Flag_{\S,\iota}(M)$ under the projection in \eqref{E:Flaghwtp}.
Using the canonical $\Diff(M)$ equivariant identifications
\begin{equation}\label{E:Flagwthwt}
	\Flag^{\hwt}_{\S,\iota}(M)=\Emb_{S_r}(M)\times_{\Diff(\S,\iota)}H(\S,\iota),
%	\Flag^{\hwt}_{\S,\iota}(M)=\pFlag_{\S,\iota}(M)\times_{\Diff(\S,\iota)}H(\S,\iota),
\end{equation}
we equip $\Flag^{\hwt}_{\S}(M)$ with the structure of a smooth vector bundle of finite (possibly nonconstant) rank over $\Flag_{\S}(M)$ and with projection~\eqref{E:Flaghwtp}.
The map in \eqref{E:hS} is a smooth bundle map over $\Flag_{\S}(M)$.
Indeed, via the identifications in \eqref{E:FlagwtSiota} and \eqref{E:Flagwthwt}, the map $h_{\S,\io}$ in \eqref{E:hiota} induces a bundle map $\Flag^{\wt}_{\S,\iota}(M)\to\Flag^{\hwt}_{\S,\iota}(M)$ which coincides with the restriction of \eqref{E:hS}.

We define the space of \emph{homologically weighted flags of type $(\S,\iota,[\mu])$ in $M$} by
\[
	\Flag^{\hwt}_{\S,\iota,[\mu]}(M):=\left\{(\mathcal N,[\nu])\in\Flag^{\hwt}_{\S}(M)\middle|\begin{array}{c}\bigl(S_1\xrightarrow{\iota_1}S_2\xrightarrow{\iota_2}\cdots\to S_r,[\mu]\bigr)\\\cong\bigl(N_1\subseteq N_2\subseteq\cdots\subseteq N_r,[\nu]\bigr)\end{array}\right\},
\]
that is, the space of all homologically weighted flags $(\mathcal N,[\nu])$ in $M$ such that there exist diffeomorphisms $S_i\to N_i$, $1\leq i\leq r$, intertwining $\iota_i$ with the canonical inclusion $N_i\subseteq N_{i+1}$ and taking $[\mu]\in H(\S,\iota)$ to $[\nu]\in H(\mathcal N)$.
The diffeomorphism in~\eqref{E:Flagwthwt} restricts to a bijection
\begin{equation}\label{E:Flaghwtasoc}
	\Flag^{\hwt}_{\mathcal S,\iota,[\mu]}(M)=\Emb_{S_r}(M)\times_{\Diff(\mathcal S;\iota)}H(\S,\io)_{[\mu]}.
%	\Flag^{\hwt}_{\mathcal S,\iota,[\mu]}(M)=\pFlag_{\mathcal S,\iota}(M)\times_{\Diff(\mathcal S;\iota)}H(\S,\io)_{[\mu]}.
\end{equation}
Hence, since the orbit $H(\S,\io)_{[\mu]}$ is finite, $\Flag^{\hwt}_{\mathcal S,\iota,[\mu]}(M)$ is a $\Diff(M)$ invariant closed submanifold in $\Flag^{\hwt}_{\mathcal S,\iota}(M)$ of codimension $\dim H(\S,\io)$
by \eqref{E:Flagwthwt}.
Moreover, the forgetful map 
\begin{equation}\label{E:covering}
	\Flag^{\hwt}_{\mathcal S,\iota,[\mu]}(M)\to\Flag_{\mathcal S,\iota}(M)
\end{equation}
is a $\Diff(M)$ equivariant covering map with finite fibers.

We complement Theorem~\ref{T:wtFlag.homog} with the following:

\begin{theorem}\label{L4}
	In this situation we have
	\[
		\Flag^{\wt}_{\S,\iota,\mu}(M)=h_{\Flag_{\S}(M)}^{-1}\bigl(\Flag^{\hwt}_{\mathcal S,\iota,[\mu]}(M)\bigr).
	\]
\end{theorem}

\begin{proof}
	This follows from the identity \eqref{densim} in Proposition~\ref{P:DenS.homog}\eqref{I:DenS.homog.orbit}, using \eqref{Simu} and \eqref{E:Flaghwtasoc}.
\end{proof}

\begin{remark}
	If the action of $\Diff(\S,\io)$ on $H(\S,\io)$ is trivial, then the $\Diff(\S,\io)_0$ and $\Diff(\S,\io)$ orbits in $\Den_\x(\S)$ coincide in view of Proposition~\ref{P:DenS.homog}, and
%	The $\Diff(\S,\io)$ orbit of $[\mu]\in H(\S,\io)$ being a single element set, 
	the forgetful (covering) map in \eqref{E:covering} is a diffeomorphism, $\Flag^{\hwt}_{\mathcal S,\iota,[\mu]}(M)=\Flag_{\mathcal S,\iota}(M)$. 
	This happens in particular when all $S_i\setminus\io(S_{i-1})$ are connected, as in Examples~\ref{nesttor} and \ref{nestpro} below.
	Under the latter connectedness assumption, $H(\S,\io)=\R^r$ via the isomorphism $[\mu]\mapsto\bigl(\int_{S_1}\mu_1,\dotsc,\int_{S_r}\mu_r\bigr)$ and
	\[
		\Den(\S)_{\io,\mu}=\bigl\{\alpha\in\Den_\x(\S):\textstyle\int_{S_i}\alpha_i=\int_{S_i}\mu_i\bigr\}.
	\]
	Moreover, Theorem~\ref{L4} ensures that
	\begin{equation}\label{a}
		\Flag^{\wt}_{\S,\iota,\mu}(M)=\bigl\{(\N,\nu)\in\Flag^{\wt}_{\mathcal S,\iota}(M):\textstyle\int_{N_i}\nu_i=\int_{S_i}\mu_i\bigr\}.
	\end{equation}	
	This also applies in the situation of Remark~\ref{larg}, provided each model manifold $S_i$ is connected.
	A description for nested spheres in the same vein can be found in Example~\ref{nest}.
\end{remark}

\begin{example}[Nested tori]\label{nesttor}
	If $(\S,\io)$ denotes the standard (meridional) embeddings between tori, $\TT^0\subseteq\TT^1\subseteq\cdots\subseteq\TT^r$, then $H(\S,\io)=\R^{r+1}$ via the isomorphism $[\mu]\mapsto(\int_{\TT^i}\mu_i)$.
\end{example}

\begin{example}[Nested projective spaces]\label{nestpro}
	If $(\S,\io)$ denotes the standard embeddings between projective spaces, $\PP^0\subseteq\PP^1\subseteq\cdots\subseteq\PP^r$, then $H(\S,\io)=\R^{r+1}$ via the isomorphism $[\mu]\mapsto(\int_{\PP^i}\mu_i)$.
	%, and the action of $\Diff(\S,\io)$ on $H(\S,\io)$ is trivial.
\end{example}

\begin{example}[{Nested spheres \cite{Jung}}]\label{nest}
	If $(\S,\io)$ denotes the standard equatorial embeddings between spheres, $S^0\subseteq S^1\subseteq\cdots\subseteq S^r$, then $H(\S,\io)=\R^{2(r+1)}$.
	The $2(r+1)$ numbers assigned to $[\mu]\in H(\S,\io)$ by this isomorphism are
	\begin{equation*}
		\textstyle
		(a_0^+,a_0^-,a_1^+,a_1^-,\dotsc,a_r^+,a_r^-)=\left(\int_{S^0_+}\mu_0,\int_{S^0_-}\mu_0,\int_{S^1_+}\mu_1,\int_{S^1_-}\mu_1,\dotsc,\int_{S^r_+}\mu_r,\int_{S^r_-}\mu_r\right),
	\end{equation*}
	where $S^i_+$ and $S^i_-$ denote the northern and southern hemispheres of $S^i$, respectively.
	Considering reflections on hyperplanes, we see that for each $0\leq i\leq r$ there exists a diffeomorphism in $\Diff(\S,\iota)$ swapping $S^i_+$ with $S^i_-$, but leaving all other hemispheres $S^k_\pm$ invariant.
	Such a diffeomorphism interchanges $a_i^+$ with $a_i^-$, but leaves all other numbers $a_k^\pm$ unchanged. 
	Hence, the $\Diff(\S,\iota)$ orbit $H(\S,\io)_{[\mu]}$ has $2^s$ elements, where $s$ is the number of $0\leq i\leq r$ with $a_i^+\neq a_i^-$.
	Actually the $\Diff(\S,\iota)$ action on $H(\S,\io)\cong\RR^{2(r+1)}$ factorizes through an $(\ZZ_2)^{2(r+1)}$ action by switching or not the numbers 
	$a_i^+$ and $a_i^-$.
	%	In particular, if $a_i^+=a_i^-=:a_i$ for all $i$, then the orbit $H(\S,\io)_{[\mu]}$ is a one-point set, and w
	We obtain a description similar to the one in \eqref{a}:
	\begin{equation}\label{asta}
		\Flag_{\S,\io,\mu}^{\wt}(M)=\left\{(\N,\nu)\in\Flag_{\S,\io}^{\wt}(M)\middle|\begin{array}{c}
		\text{
		$\{\int_{N_i^+}\nu_i,\int_{N_i^-}\nu_i\}=\{a_i^+,a_i^-\}$, 
		where $N_i^+$ and $N_i^-$}\\\text{denote the connected components of  $N_i\setminus N_{i-1}$}\end{array}\right\}.
	\end{equation}
\end{example}

%\todo{I tried to formulate a more general version of the last remark, but the outcome was not very satisfying}		
%%%%%%%%%%%%%%%%%\newpage

\section{Coadjoint orbits  of the Hamiltonian group}
% consisting of nonlinear flags}

Throughout this section $(M,\om)$ denotes a symplectic manifold.
The nonlinear Grassmannian of all isotropic submanifolds of type $S_r$ in $M$, denoted by $\Gr_{S_r}^{\iso}(M)$, is a splitting smooth submanifold in $\Gr_{S_r}(M)$ which is invariant under the Hamiltonian group.
In fact the $\Ham(M)$ orbits provide a smooth foliation of finite codimension in $\Gr_{S_r}^{\iso}(M)$ which is called the isodrastic foliation \cite{W90,Lee}.

Suppose $\L$ is an isodrastic leaf in $\Gr_{S_r}^{\iso}(M)$ and let $\Flag^{\wt\iso}_{\S,\iota,\mu}(M)|_\L$ denote the preimage of $\L$ under the canonical bundle projection $\Flag^{\wt}_{\S,\iota,\mu}(M)\to\Gr_{S_r}(M)$.
We will show that the natural $\Ham_c(M)$ action on $\Flag^{\wt\iso}_{\S,\iota,\mu}(M)|_\L$ admits local smooth sections. 
In particular, each connected component of the latter space is an orbit of $\Ham_c(M)$.

The space $\Flag^{\wt\iso}_{\S,\iota,\mu}(M)|_\L$ comes equipped with a canonical weakly non-degenerate symplectic form, and the restriction of \eqref{pl} provides an $\Ham(M)$ equivariant injective moment map for the $\Ham_c(M)$ action,
\[
	J\colon\Flag^{\wt\iso}_{\mathcal S,\iota,\mu}(M)|_\L\hookrightarrow\ham_c(M)^*,\quad
	\langle J(\mathcal N,\nu),X_f\rangle=\sum_{i=1}^r\int_{N_i}f\nu_i.
\]
This moment map $J$ maps each connected component of $\Flag^{\wt\iso}_{\mathcal S,\iota,\mu}(M)|_\L$ one-to-one onto the corresponding coadjoint orbit, see Theorem~\ref{t2}.
Thereby, we identify coadjoint orbits of the Hamiltonian group $\Ham_c(M)$ that can be modeled on weighted nonlinear flags.

The material in this section is inspired by the results in \cite{W90,Lee,GBV17} on weighted isotropic nonlinear Grassmannians.

%%%%%%%%%%%%%%%%%%%%%%%

\subsection{Isodrasts as $\Ham(M)$ orbits}\label{3.1}

In view of the tubular neighborhood theorem for isotropic embeddings \cite{W77,W81}, the space $\Gr_S^{\iso}(M)$ of all isotropic submanifolds of type $S$ in $M$ is a splitting smooth submanifold of $\Gr_S(M)$, see for instance \cite[Section~8]{Lee}.
The tangent space at an isotropic submanifold $N$ is
\begin{equation*}%\label{tison}
	T_N\Gr_S^{\iso}(M)
	=\{u_N\in\Ga(TN^\perp)| \io_N^*i_{u_N}\om\in\Om^1(N) \text{ closed}\},
\end{equation*}
where $TN^\perp:=TM|_N/TN$ denotes the normal bundle to $N$.

Weinstein's \cite{W90} \emph{isodrastic distribution} $\underline\D$ on $\Gr_S^{\iso}(M)$ is given by
\begin{equation}\label{DDDDn}
\underline\D_N:=\{u_N\in\Ga(TN^\perp)| \io_N^*i_{u_N}\om\in dC^\oo(N)\}
\end{equation}
and has finite codimension $\dim H^1(S;\RR)$.
This is an integrable distribution \cite{W90,Lee} whose leaves are orbits of $\Ham(M)$.
It gives rise to a smooth foliation of $\Gr^{\iso}_S(M)$ called the \emph{isodrastic foliation.}
In particular, the $\Ham(M)$ orbits in $\Gr^{\iso}_S(M)$ are splitting smooth submanifolds.

Restricting the fundamental frame bundle in \eqref{E:frame}, we obtain a principal $\Diff(S)$ bundle $\Emb^{\iso}_{S}(M)\to\Gr_S^{\iso}(M)$ with total space $\Emb^{\iso}_{S}(M)$, the splitting smooth submanifold of $\Emb_{S}(M)$ consisting of all isotropic embeddings.
On $\Emb_S^{\iso}(M)$ we consider the pullback of the  isodrastic distribution $\underline\D$:
\begin{equation}\label{DDDD}
	\mathcal{D}_\ph:=\bigl\{u_\ph\in\Ga(\ph^*TM):\ph^*i_{u_\ph}\om\in dC^\oo(S)\bigr\}.
\end{equation}
This is an integrable distribution with the same codimension, $\dim H^1(S;\RR)$, and the leaves of $\D$ are connected components in the preimage of a leaf of $\underline\D$.
According to \cite{W90,Lee}, the group $\Ham(M)$ acts transitively on the leaves of $\D$.
We need the subsequent slightly stronger statement in Proposition~\ref{lss}.

The Lie algebra of compactly supported Hamiltonian vector fields will be denoted by $\ham_c(M)=\{X_f:f\in C^\infty_c(M)\}$.
We let $\Ham_c(M)$ denote the group of diffeomorphisms obtained by integrating time dependent vector fields in $\ham_c(M)$.
For our purpose it will not be necessary to consider $\Ham_c(M)$ as an infinite dimensional Lie group, cf.~\cite[Section~43.13]{KM}.
By a smooth map (section) into $\Ham_c(M)$ we will simply mean a smooth map into $\Diff_c(M)$ that takes values in $\Ham_c(M)$. 

\begin{proposition}\label{lss}
	The $\Ham_c(M)$ action on each leaf $\E\subseteq\Emb^{\iso}_S(M)$ of the isodrastic foliation $\D$ admits local smooth sections.
\end{proposition}

\begin{proof}
	To show infinitesimal transitivity, suppose $\varphi\in\Emb_S^{\iso}(M)$ and $u_\ph\in\mathcal D_\ph$, cf.~\eqref{DDDD}.
	Hence, there exists $\bar f\in C^\oo(S)$ such that $d\bar f=\ph^*i _{u_{\ph}}\om$.
	We extend $\bar f$ to a function $f_1\in C^\oo_c(M)$ such that $\bar f=f_1\o \ph$.
	The 1-form $\be$ on $M$ along $S$ defined by
	\begin{equation}\label{E:beta}
		\be=d f_1\o \ph-i _{u_{\ph}}\om\in\Ga(\ph^*T^*M)
	\end{equation}
	vanishes on vectors tangent to $\ph(S)\subseteq M$.
	Hence, $\beta$ can be seen as a fiberwise linear function on the normal bundle $T\ph(S)^\perp$ whose differential along the zero section $\ph(S)$ coincides with $\beta$ itself.
	Thus, with the help of a tubular neighborhood of $\ph(S)$ in $M$ and a suitable bump function, we get $f_2\in C^\oo_c(M)$ such that $\be=d f_2\o \ph$.
	It follows from \eqref{E:beta} that $df\o \ph=i _{u_{\ph}}\om$ for $f=f_1-f_2$.
	We conclude that $u_{\ph}=X_{f}\o \ph$ is the infinitesimal generator at $\ph$ for the Hamiltonian vector field $X_f\in\ham_c(M)$.

	Using tubular neighborhoods constructed with the help of a Riemannian metric, say, we see that the function $f$ may be chosen to depend smoothly on $\ph$ und $u_\ph$, for $\ph$ in a sufficiently small open neighborhood of a fixed isotropic embedding $\ph_0\in\E$. 
	Hence, we may apply Lemma~\ref{L:loc.sec} and conclude that the $\Ham_c(M)$ action admits local smooth sections.
\end{proof}

\begin{corollary}\label{adun}
	The $\Ham_c(M)$ action on each leaf $\L\subseteq\Gr^{\iso}_S(M)$ of the isodrastic foliation $\underline\D$ admits local smooth sections.
\end{corollary}

\begin{example}\label{drast}
Every embedded closed curve in the plane is a Lagrangian submanifold  of $(\RR^2,\om)$, where $\om$ is the canonical area form,
thus an element of $\Gr^{\iso}_{S^1}(\RR^2)$. The isodrastic distribution $\underline\D$ has codimension one.
The enclosed area $a$ singles out one isodrast $\L_a\subseteq\Gr^{\iso}_{S^1}(\RR^2)$, \ie one orbit of  $\Ham_c(\RR^2)$.

A similar phenomena happens for Lagrangian $k$-tori in $\RR^{2k}$, \ie elements of $\Gr^{\iso}_{\TT^k}(\RR^{2k})$,
where  $\TT^k:=(S^1)^k$.
To any $\ph\in\Emb_{\TT^k}^{\iso}(\RR^{2k})$ we assign the symplectic area $a_i$ of the surface in $\RR^{2k}$ enclosed by the $i$th meridian 
$\ph_i(\th)=\ph(1,\dots,\th,\dots,1)$ of the embedded $k$-torus.
These numbers are independent of the choice of the meridian in its homotopy class and of the surface having the meridian as boundary.
The $k$-tuple $(a_1,\dots,a_k)$ is an invariant under isodrastic deformations. 
Actually $a_i$ is the action integral of the $i$th meridian, as defined in  \cite{W90}.

Let $\E_{a_1,\dots,a_k}\subseteq\Emb^{\iso}_{\TT^k}(\RR^{2k})$ be the space of all isotropic embeddings having symplectic areas $a_1,\dots,a_k$.
It is a union of isodrastic leaves of Lagrangian embeddings, but it is not necessarily $\Diff(\TT^k)$ saturated.
The $\Diff(\TT^k)$ action on the $k$-tuples $(a_1,\dots,a_k)$ factorizes through an $\SL(k,\ZZ)$ action.
Let $[a_1,\dots,a_k]$ denote the orbit of $(a_1,\dots,a_k)$.
We define $\L_{[a_1,\dots,a_k]}\subseteq\Gr^{\iso}_{\TT^k}(\RR^{2k})$ to be the image of $\E_{a_1,\dots,a_k}$ under the principal $\Diff(\TT^k)$ bundle projection \eqref{E:frame}.
Thus $\L_{[a_1,\dots,a_k]}$ is a union of isodrastic leaves of Lagrangian $k$-tori in $\RR^{2k}$.

There is also a direct description of $\L_{[a_1,\dots,a_k]}$. Given a Lagrangian torus $N\in\Gr_{\TT^k}^S(\RR^{2k})$, we choose a base 
$\{[\ga_1],\dots,[\ga_k]\}$ of $H_1(N,\ZZ)$, where $\ga_i$ are loops in $N$ with action integrals $a_i$.
We observe that the $\SL(k,\ZZ)$ orbit  $[a_1,\dots,a_k]$ is independent of the choices
and $\L_{[a_1,\dots,a_k]}$ is the space of all Lagrangian tori in $\Gr_{\TT^k}^S(\RR^{2k})$ such that the orbit of these action integrals is 
$[a_1,\dots,a_k]$.
\end{example}

We will also need the following observation.

\begin{lemma}\label{L:GrSNMiso}
	If $L$ is an isotropic submanifold in $M$, then the canonical inclusion $\Gr_S(L)\subseteq\Gr_S^{\iso}(M)$ is a splitting smooth submanifold.
	Moreover, each connected component of $\Gr_S(L)$ is a splitting smooth submanifold in an isodrastic leaf in $\Gr_S^{\iso}(M)$.
\end{lemma}

\begin{proof}
	Suppose $N\in\Gr_S(L)$, i.e., $N\cong S$ is a closed submanifold in $L$.
	By the tubular neighborhood theorem, we may w.l.o.g.\ assume that $L$ is the total space of a vector bundle $p\colon L\to N$, the normal bundle of $N$ in $L$, and identify $N$ with the zero section in $L$.
	We have a canonical short exact sequence $0\to p^*L\to TL\to p^*TN\to0$ of vector bundles over $L$.
	Choosing a linear connection on $L$, we obtain a splitting of this sequence and thus an isomorphism $TL\cong p^*TN\oplus p^*L$ of vector bundles over $L$.
	Dualizing, we obtain an isomorphism $T^*L\cong p^*T^*N\oplus p^*L^*$ of vector bundles over $L$.
	We regard this as a diffeomorphism \[T^*L\cong T^*N\oplus L\oplus L^*\] that maps the zero section $L\subseteq T^*L$ identically onto the summand $L$ on the right hand side. 
	Via this isomorphism we have \[\theta_L=\pi_1^*\theta_N+\pi_2^*\kappa,\] where $\pi_1$ and $\pi_2$ denote the projections from $T^*N\oplus L\oplus L^*$ onto $T^*N$ and $L\oplus L^*$, respectively, $\theta_L\in\Omega^1(T^*L)$ and $\theta_N\in\Omega^1(T^*N)$ denote the tautological 1-forms, and $\kappa\in\Omega^1(L\oplus L^*)$.
	Indeed, a straightforward computation yields $\kappa(\xi)=\ell'(C(Tq\cdot\xi))$ where $x\in N$, $\ell\in L_x$, $\ell'\in L_x^*$, $\xi\in T_{(\ell,\ell')}(L\oplus L^*)$, $q\colon L\oplus L^*\to L$ denotes the projection and $C$ denotes the linear connection on $L$, viewed as a fiberwise linear map $C\colon TL\to p^*L$ over $L$.

	By the tubular neighborhood theorem for isotropic embeddings \cite{W77,W81}, we may assume $M=T^*L\oplus p^*E$ and $\omega=\tilde\pi_1^*d\theta_L+\tilde\pi_2^*\rho$, where $E$ is a vector bundle over $N$, $\tilde\pi_1$ and $\tilde\pi_2$ denote the projections from $T^*L\oplus p^*E$ onto $T^*L$ and $p^*E$, respectively, and $\rho$ is a closed 2-form on the total space of $p^*E$.
	Combining this with the diffeomorphism in the previous paragraph, we obtain a diffeomorphism 
	\begin{equation}\label{E:diffeo}
		M=T^*N\oplus L\oplus L^*\oplus E,
	\end{equation}
	mapping $L\subseteq M$ identically onto the summand $L$ on the right hand side, and such that  \[\omega=\pi_1^*d\theta_N+\pi_2^*\sigma,\]
	where $\pi_1$ and $\pi_2$ denote the projections from $T^*N\oplus L\oplus L^*\oplus E$ onto $T^*N$ and $L\oplus L^*\oplus E$, respectively, and $\sigma$ is a closed 2-form on the total space of $L\oplus L^*\oplus E$.

	The diffeomorphism in \eqref{E:diffeo} provides a (standard) chart for the smooth structure on $\Gr_S(M)$ centered at $N$, \[\Gamma(T^*N\oplus L\oplus L^*\oplus E)\to\Gr_S(M),\qquad\phi\mapsto\phi(N).\]
	In this chart, the inclusions $\Gr_S(L)\subseteq\Gr_S^{\iso}(M)\subseteq\Gr_S(M)$ become 
	\begin{equation}\label{E:incl}
		\Gamma(L)\subseteq\{(\alpha,\xi)\in\Omega^1(N)\times\Gamma(L\oplus L^*\oplus E):d\alpha+\xi^*\sigma=0\}\subseteq\Omega^1(N)\times\Gamma(L\oplus L^*\oplus E).
	\end{equation}
	As $L$ is isotropic, $\sigma$ vanishes when pulled back to $L$.
	Hence, by the Poincar\'e lemma, $\sigma=d\beta$ for a 1-form $\beta$ on the total space of $L\oplus L^*\oplus E$ which vanishes along $L$.
	Via the diffeomorphism of $\Omega^1(N)\times\Gamma(L\oplus L^*\oplus E)$ given by $(\alpha,\xi)\mapsto(\alpha-\xi^*\beta,\xi)$, the inclusions in \eqref{E:incl} become linear inclusions,
	\[
		\Gamma(L)\subseteq Z^1(N)\times\Gamma(L\oplus L^*\oplus E)\subseteq\Omega^1(N)\times\Gamma(L\oplus L^*\oplus E),
	\]
	where $Z^1(N)$ denotes the space of closed 1-forms on $N$.
	Clearly, both inclusions admit complementary subspaces.
	In particular, $\Gr_S(L)$ is a splitting smooth submanifold in $\Gr_S^{\iso}(M)$.
	The second assertion follows from the fact that the isodrastic leaf through $N$ corresponds to the subspace 
	$B^1(N)\times\Gamma(L\oplus L^*\oplus E)$, see \cite[Section 8]{Lee}.
\end{proof}

\subsection{Weighted isotropic nonlinear Grassmannians as coadjoint orbits}\label{3.2}

In this section we recall the results in \cite{W90,Lee,GBV17} about coadjoint orbits of the Hamiltonian group $\Ham_c(M)$ modeled on weighted isotropic nonlinear Grassmannians of type $S$ in $M$,
and extend them to a possibly nonconnected model manifold $S$. 
Here we present them in a manner that readily generalizes to manifolds of weighted nonlinear flags.

Let $S$ be a closed $k$-dimensional manifold. 
The preimage of $\Gr^{\iso}_S(M)$ under the canonical bundle projection $\Gr_S^{\wt}(M)\to\Gr_S(M)$ is a splitting smooth submanifold in $\Gr_S^{\wt}(M)$ which will be denoted by $\Gr_{S}^{\wt\iso}(M)$.
The diffeomorphism in \eqref{grsw} restricts to a diffeomorphism of bundles over $\Gr_S^{\iso}(M)$, 
\[\Gr_{S}^{\wt\iso}(M)=\Emb_S^{\iso}(M)\x_{\Diff(S)}\Den_\x(S).\]

The preimage of $\Gr^{\iso}_S(M)$ under the canonical bundle projection $\Gr_{S,\mu}^{\wt}(M)\to\Gr_S(M)$ is a splitting smooth submanifold in $\Gr_{S,\mu}^{\wt}(M)$ which will be denoted by $\Gr_{S,\mu}^{\wt\iso}(M)$ and will be referred to as the \emph{nonlinear Grassmannian of weighted isotropic submanifolds of type $(S,\mu)$} in $M$.
The diffeomorphism in \eqref{E:grswmu} restricts to a diffeomorphism of bundles over $\Gr_S^{\iso}(M)$, 
\[
\Gr_{S,\mu}^{\wt\iso}(M)=\Emb^{\iso}_S(M)\x_{\Diff(S)}\Den(S)_\mu.
\]

We recall the $\Diff(S)$ equivariant linear map  $h_S:\Den(S)\to H^k(S,\O_S)$,
$h_S(\al)=[\al]$ in \eqref{E:h}, with kernel $d\Om^{k-1}(S,\O_S)$, which we restrict to $\Den_{\x}(S)$.
The product of integrable distributions $\D\x\ker Th_S$  on $\Emb_S^{\iso}(M)\x\Den_\x(S)$
descends to an integrable distribution 
\[
\bar\D:=\D\x_{\Diff(S)}\ker Th_S
\]
on $\Gr_{S}^{\wt\iso}(M)$, of codimension $\dim H^1(S;\RR)+\dim H^k(S;\O_S)$.
The image of $\bar\D$ under the forgetful map
$\Gr_S^{\wt\iso}(M)\to\Gr_S^{\iso}(M)$ is the isodrastic distribution $\underline\D$ in \eqref{DDDDn}.
In view of~\eqref{E:DenSmu}, each leaf $\G$ of $\bar\D$ is a connected component in the preimage of an isodrast $\L$ in $\Gr_S^{\iso}(M)$ under the bundle projection $\Gr_{S,\mu}^{\wt\iso}(M)\to\Gr_S^{\iso}(M)$, for some volume density $\mu$ on $S$.
Hence, each leaf $\G$ of $\bar\D$ is a splitting smooth submanifold of codimension $\dim H^1(S;\RR)$ in $\Gr_{S,\mu}^{\wt\iso}(M)$.

According to \cite{W90,Lee}, the group $\Ham(M)$ acts transitively on the leaves of $\bar\D$.
We need the following slightly stronger statement.

\begin{proposition}\label{abo}
	The $\Ham_c(M)$ action on each leaf $\G\subseteq\Gr^{\wt\iso}_S(M)$ of the isodrastic foliation $\bar\D$ admits local smooth sections.
\end{proposition}

\begin{proof}
	Suppose $\L$ is an isodrastic leaf in $\Gr_S^{\iso}(M)$ and let $\Gr_{S,\mu}^{\wt\iso}(M)|_{\L}$ denote its preimage under the bundle projection $\Gr_{S,\mu}^{\wt\iso}(M)\to\Gr_S^{\iso}(M)$.
	It suffices to show that the $\Ham_c(M)$ action on $\Gr_{S,\mu}^{\wt\iso}(M)|_{\L}$ admits local smooth sections.
	The diffeomorphism in \eqref{E:grswmu} restricts to a diffeomorphism 
	\begin{equation}\label{aso}
		\Gr_{S,\mu}^{\wt\iso}(M)|_{\L}=\Emb^{\iso}_S(M)|_{\L}\x_{\Diff(S)}\Den(S)_\mu.
	\end{equation}
	By Proposition~\ref{lss}, the $\Ham_c(M)$ action on $\Emb^{\iso}_S(M)|_\L$ admits local smooth sections.
	According to Proposition~\ref{P:DenS.homog}(b), the $\Diff(S)$ action on $\Den(S)_\mu$ admits local smooth sections too. 
	Using Lemma~\ref{before.sec} in the Appendix, we conclude that the $\Ham_c(M)$ action on the associated bundle in \eqref{aso} admits local smooth sections.
\end{proof}

\begin{lemma}[{\cite{Lee}}]\label{noua}
The leafwise differential 2-form on $(\Emb_S^{\iso}(M)\x\Den_\x(S),\D\x\ker Th_S)$, given by
\begin{equation}\label{omphal}
\Om_{(\ph,\al)}((u_\ph,d\la),(v_\ph,d\ga))
:=\int_S(\om(u_\ph,v_\ph)\al-\ph^*i_{u_\ph}\om\wedge\ga
+\ph^*i_{v_\ph}\om\wedge\la),
\end{equation}
is closed and $\Diff(S)$ invariant.
Moreover, its kernel is spanned by the infinitesimal generators of the $\Diff(S)$ action.
\end{lemma}

By Lemma~\ref{noua}, the leafwise differential 2-form $\Om$ in \eqref{omphal} descends to a leafwise symplectic form $\bar\Om$ on $(\Gr_S^{\wt\iso}(M),\bar\D)$.
Thus every leaf $\G$ of the isodrastic distribution $\bar\D$ in $\Gr_S^{\wt\iso}(M)$ is endowed with a symplectic form, the restriction of $\bar\Om$, which we denote by the same letter.
By Proposition~\ref{abo}, $\G$ is an orbit for the $\Ham_c(M)$ action on the weighted isotropic nonlinear Grassmannian $\Gr_S^{\wt\iso}(M)$. 
This action is Hamiltonian with injective and $\Ham(M)$ equivariant moment map \cite{Lee}
\begin{equation}\label{bird}
	J:\G\subseteq\Gr_S^{\wt\iso}(M)\to\ham_c(M)^*,\quad \langle J(N,\nu),X_f\rangle=\int_Nf\nu.
\end{equation}
Here we use the $\Symp(M)$ equivariant isomorphism $\ham_c(M)=C_0^\oo(M)$ where the latter denotes the Lie algebra of all compactly supported functions on $M$ for which the integral with respect to the Liouville form vanishes on all closed connected components of $M$.

For connected $S$, the subsequent theorem is due to Weinstein \cite{W90} in the Lagrangian case and due to Lee \cite{Lee} for isotropic submanifolds.

\begin{theorem}[{\cite{W90,Lee}}]\label{t1}
	The moment map $J\colon(\G,\bar\Om)\to\ham_c(M)^*$ in \eqref{bird} is one-to-one onto a coadjoint orbit of the Hamiltonian group $\Ham_c(M)$. 
	The Kostant-Kirillov-Souriau symplectic form $\om_{\KKS}$ on the coadjoint orbit satisfies ${J}^*\om_{\KKS}=\bar\Om$.
\end{theorem}

In this generality the theorem follows from the discussion above and the following folklore result (see for instance the Appendix in \cite{HV20}):

\begin{proposition}\label{spcase}
	Suppose the action of $G$ on $(\mathcal M,\Om)$ is transitive with injective equivariant moment map $J:\mathcal M\to\g^*$.
	Then $J$ is one-to-one onto a coadjoint orbit of $G$.
	Moreover, it pulls back the Kostant--Kirillov--Souriau symplectic form $\omega_{\KKS}$ on the coadjoint orbit to the symplectic form $\Om$.
\end{proposition}

The same coadjoint orbits of the Hamiltonian group,
under the additional restriction $H^1(S;\RR)=0$ on the closed connected manifold $S$,
can be obtained via symplectic reduction in the Marsden-Weinstein ideal fluid dual pair, as shown in \cite{GBV17}.

%%%%%%%%%%%%%%%%%%%

\subsection{Manifolds of isotropic nonlinear flags}\label{SS:manf}

In this section we extend the constructions of the previous section to nonlinear flags.
%For this, isotropic versions of the manifolds of (weighted) nonlinear flags in a symplectic manifold $(M,\om)$ are needed. 

As in Section~\ref{SS:red}, we let $\iota=(\iota_1,\dotsc,\iota_{r-1})$ denote a collection of fixed embeddings $\iota_i\colon S_i\to S_{i+1}$.
We let $\Flag_{\S,\io}^{\iso}(M)$ denote the preimage of $\Gr_{S_r}^{\iso}(M)$ under the canonical bundle projection $\Flag_{\S,\io}(M)\to\Gr_{S_r}(M)$, cf.~\cite[Remark~2.11]{HV20}.
This is a splitting smooth submanifold in $\Flag_{\S,\io}(M)$ called the manifold of \emph{isotropic nonlinear flags of type $(\S,\io)$ in $M$.}
Using Lemma~\ref{L:GrSNMiso}, and proceeding by induction on the depth of the flag, one readily shows that the canonical inclusion
\begin{equation}\label{E:FlagGriso}
	\Flag_{\S,\io}^{\iso}(M)\subseteq\prod_{i=1}^r\Gr_{S_i}^{\iso}(M)
\end{equation}
is a splitting smooth submanifold.

Let $\Flag_{\S,\io}^{\wt\iso}(M)$ denote the preimage of $\Flag_{\S,\io}^{\iso}(M)$ under the canonical bundle projection $\Flag_{\S,\io}^{\wt}(M)\to\Flag_{\S,\io}(M)$.
This is a splitting smooth submanifold in $\Flag_{\S,\io}^{\wt}(M)$ called the manifold of \emph{weighted isotropic nonlinear flags of type $(\S,\io)$ in $M$.}
The diffeomorphism in \eqref{E:FlagwtSiota} restricts to a diffeomorphism of bundles over $\Flag_{\S,\io}^{\iso}(M)$,
\[
	\Flag_{\S,\io}^{\wt\iso}(M)=\Emb_{S_r}^{\iso}(M)\x_{\Diff(\S;\io)}\Den_\x(\S).
\]
The canonical inclusion $\Flag_{\S,\io}^{\wt\iso}(M)\subseteq\prod_{i=1}^r\Gr_{S_i}^{\wt\iso}(M)$ is a splitting smooth submanifold in view of \eqref{E:FlagGriso}.

As in Section~\ref{SS:DiffM.Flag}, we fix $\mu=(\mu_1,\dotsc,\mu_r)\in\Den_\x(\S)$, and let $\Flag_{\S,\io,\mu}^{\wt\iso}(M)$ denote the preimage of $\Flag_{\S,\io}^{\iso}(M)$ under the canonical bundle projection $\Flag_{\S,\io,\mu}^{\wt}(M)\to\Flag_{\S,\io}(M)$.
This is a splitting smoth submanifold in $\Flag_{\S,\io,\mu}^{\wt}(M)$ called the manifold of \emph{weighted isotropic nonlinear flags of type $(\S,\io,\mu)$ in $M$.}
The diffeomorphism in \eqref{Simu} restricts to a diffeomorphism
\begin{equation}\label{fligs}
	\Flag_{\S,\io,\mu}^{\wt\iso}(M)=\Emb^{\iso}_{S_r}(M)\x_{\Diff(\S;\io)}\Den(\S)_{\io,\mu}.
\end{equation}
Using \eqref{E:FlagGriso} and proceeding as in the proof of Theorem~\ref{T:wtFlag.homog}(c), one readily shows that the canonical inclusion $\Flag_{\S,\io,\mu}^{\wt\iso}(M)\subseteq\prod_{i=1}^r\Gr_{S_i,\mu_i}^{\wt\iso}(M)$ is a splitting smooth submanifold.

Recall the $\Diff(\S;\io)$ equivariant linear map $h_{\S,\io}:\Den(\S)\to H(\S,\io)$ in \eqref{E:hiota}
with kernel 
\[
	\ker Th_{\S,\io}=\{(d\ga_1,\dotsc,d\ga_r):\ga_i\in\Om^{k_i-1}(S_i;\mathcal O_{S_i}),\iota_{i-1}^*\ga_i=0\},
\]
whose restriction to $\Den_\x(\S)$ we also denote by $h_{\S,\io}$.
%By Proposition \ref{P:DenS.homog}(c), the leaves of the distribution $\ker Th_{\S,\io}$ are the connected components of the $\Diff(\S;\io)$ orbits in $\Den_\x(\S)$.
The product with the isodrastic distribution gives the integrable distribution $\D\x\ker Th_{\S,\io}$ on $\Emb_{S_r}^{\iso}(M)\x\Den_\x(\S)$, of finite codimension $\dim H^1(S_r;\RR)+\dim H(\S,\io)$. 
This $\Diff(\S;\io)$ invariant distribution descends to an integrable distribution
\begin{equation*}
	\bar\D =\D\x_{\Diff(\S;\io)}\ker Th_{\S,\io}
\end{equation*}
of the same codimension on $\Flag_{\S,\io}^{\wt\iso}(M)$.
The image of $\bar\D$ under the forgetful map is an integrable distribution on $\Flag_{\S,\io}^{\iso}(M)$ of codimension $\dim H^1(S_r;\RR)$, which coincides with the distribution that descends from the $\Diff(\S;\io)$ invariant isodrastic distribution $\D$ on $\Emb_{S_r}^{\iso}(M)$ by the principal bundle projection.
%Its leaves are called \emph{isodrasts of nonlinear flags}.
Using Proposition \ref{P:DenS.homog}(c) and \eqref{fligs}, we see that each leaf $\F$ of $\bar\D$ is a connected component in the preimage of an isodrast $\L$ in $\Gr_{S_r}^{\iso}(M)$ under the bundle projection $\Flag_{\S,\io,\mu}^{\wt\iso}(M)\to\Gr_{S_r}^{\iso}(M)$, for some volume density $\mu$ on $\S$.
Hence, each leaf $\F$ of $\bar\D$ is a splitting smooth submanifold of codimension $\dim H^1(S_r;\RR)$ in $\Flag_{\S,\io,\mu}^{\wt\iso}(M)$.

\begin{remark}\label{ffff}
	If $H^1(S_r;\RR)=0$, then the leaves of the isodrastic foliation $\bar\D$ are the connected components of $\Flag_{\S,\io,\mu}^{\wt\iso}(M)$.
\end{remark}
%\end{corollary}

\begin{proposition}\label{abo2}
	The $\Ham_c(M)$ action on each leaf $\F\subseteq\Flag_{\S,\io}^{\wt\iso}(M)$ of the isodrastic foliation $\bar\D$ admits local smooth sections. 
\end{proposition}

\begin{proof}
	Suppose $\L$ is an isodrastic leaf in $\Gr^{\iso}_{S_r}(M)$ and let $\Flag_{S,\io,\mu}^{\wt\iso}(M)|_{\L}$ denote its preimage under the canonical projection $\Flag_{S,\io,\mu}^{\wt}(M)\to\Gr_{S_r}(M)$.
	It suffices to show that the $\Ham_c(M)$ action on $\Flag_{S,\io,\mu}^{\wt\iso}(M)|_{\L}$ admits local smooth sections.
	The diffeomorphism in \eqref{fligs} restricts to a diffeomorphism
	\begin{equation}\label{aso2}
		\Flag_{S,\io,\mu}^{\wt\iso}(M)|_{\L}=\Emb^{\iso}_{S_r}(M)|_{\L}\x_{\Diff(\S;\io)}\Den(\S)_{\io,\mu}.
	\end{equation}
	By Proposition~\ref{lss}, the $\Ham_c(M)$ action on $\Emb^{\iso}_{S_r}(M)|_\L$ admits local smooth sections.
	According to Proposition~\ref{P:DenS.homog}(c), the $\Diff(\S;\io)$ action on $\Den(\S)_{\io,\mu}$ admits local smooth sections too.
	With the help of Lemma~\ref{before.sec}, we conclude that the $\Ham_c(M)$ action on the associated bundle \eqref{aso2} admits local smooth sections. 
\end{proof}

\begin{corollary}\label{adunb}
	The $\Ham_c(M)$ action on each leaf of the isodrastic foliation on $\Flag^{\iso}_{\S,\iota}(M)$ admits local smooth sections.
\end{corollary}

\subsection{Weighted isotropic nonlinear flag manifolds as coadjoint orbits}

In this section we describe coadjoint orbits of the Hamiltonian group consisting of weighted isotropic nonlinear flags.

We aim at defining a leafwise symplectic form $\bar\Om$ on $(\Flag_{\S,\io}^{\wt\iso}(M),\bar\D)$.
%=(\Emb_{S_r}^{\iso}(M)\x_{\Diff(\S,\io)}\Den_\x(\S),\D\x_{\Diff(\S,\io)}\ker Th_{\S,\io}).
We start with leafwise differential 2-forms $\Om_i$ on $(\Emb_{S_i}^{\iso}(M)\x\Den_\x(S_i),\D_i\x\ker Th_{S_i})$, for $i=1,\dotsc,r$, defined as in Lemma~\ref{noua}. 
Let $j_i:=\io_{r-1}\o\cdots\o\io_i\in\Emb(S_i,S_r)$.
We consider 
\[
	q_i:=j_i^*\x\pr_i:\Emb_{S_r}^{\iso}(M)\x\Den_\x(\S)\to\Emb_{S_i}^{\iso}(M)\x\Den_\x(S_i),
\]
%for $\pr_i((\al_i))=\al_i$ the projection on the $i$th factor,
which maps the distribution $\D\x\ker Th_{\S,\io}$ to $\D_i\x\ker Th_{S_i}$.
Then
\begin{equation}\label{oioi}
	\Om:=\sum_{i=1}^rq_i^*\Om_i
\end{equation}
is a closed leafwise differential 2-form on $(\Emb_{S_r}^{\iso}(M)\x\Den_\x(\S),\D\x\ker Th_{\S,\io})$.

\begin{remark}
	The image under $Tq_i$ of the distribution $\D\x\ker Th_{\S,\io}$ is, in general, strictly included in the distribution $\D_i\x\ker Th_{S_i}$.
	The reason is that  the projection on the $i$th factor $\pr_i:\ker Th_{\S,\io}\to\ker Th_{S_i}$ is not surjective in general.
	This happens for instance in the setting of Example \ref{3,14}, where $\io:\{1,\dots,k\}\to S^1$ with $k\ge 2$:
	the projection on the second factor is $d\{\ga\in C^\oo(S^1):\ga\o\io=0\}$, strictly included in $\ker Th_{S^1}=dC^\oo(S^1)$.
\end{remark}

The $\Diff(\S;\io)$ action on $\Emb_{S_r}^{\iso}(M)\x\Den_\x(\S)$,
\begin{equation}\label{gdot}
	g\cdot(\ph,(\al_i))=(\ph\o g_r,(g_i^*\al_i)),\text{ for }g=(g_1,\dotsc,g_r)\in\Diff(\S;\io),
\end{equation}
has infinitesimal generators of the form
\begin{equation}\label{zizi}
	\zeta_Z(\ph,(\al_i))= (T\ph\o Z_r,(di_{Z_i}\al_i)),\text{ for }Z=(Z_1,\dotsc,Z_r)\in\X(\S;\io).
\end{equation}
They belong to the integrable distribution $\D\x\ker Th_{\S,\io}$.
Notice that $q_i=j_i^*\x\pr_i$ is equivariant over the homomorphism $g\in\Diff(\S;\io)\mapsto g_i\in\Diff(S_i)$, because $j_i\o g_i=g_r\o j_i$ for all $i$. 
Now, from the $\Diff(S_i)$ invariance of $\Om_i$ and \eqref{oioi}, we deduce the  $\Diff(\S;\io)$ invariance of $\Om$.

The next lemma is a generalization of the Lemma \ref{noua}:

\begin{lemma}\label{flg}
	The  kernel of the leafwise differential form $\Om$ is spanned by the infinitesimal generators of the $\Diff(\S;\io)$ action \eqref{gdot} on $\Emb_{S_r}^{\iso}(M)\x\Den_\x(\S)$.
\end{lemma}

\begin{proof}
The formula for the leafwise 2-form in Lemma \ref{noua} provides a similar formula for $\Om$:
\begin{multline}
\Om_{(\ph,(\al_i))}((u_\ph,(d\la_i)),(v_\ph,(d\ga_i)))
=\sum_{i=1}^r\int_{S_i}j_i^*\om(u_\ph,v_\ph)\al_i\\
-\sum_{i=1}^r\int_{S_i}j_i^*(\ph^*i_{u_\ph}\om)\wedge\ga_i
+\sum_{i=1}^r\int_{S_i}j_i^*(\ph^*i_{v_\ph}\om)\wedge\la_i.
\end{multline}
The contraction of $\Om$ with an infinitesimal generator $\ze_Z$, given in \eqref{zizi}, vanishes for all $Z\in\X(\S;\io)$.
This follows from the analogous statement for $\Om_i$ and the infinitesimal generators $\ze_{Z_i}$ on $\Emb_{S_i}^{\iso}(M)\x\Den_\x(S_i)$, 
together with the fact that the infinitesimal generators $\ze_Z$ and $\ze_{Z_i}$ are $q_i$ related.

For the converse statement, we start with an arbitrary tangent vector 
\[
	(u_\ph,(d\la_i))\in\ker\Om_{(\ph,(\al_i))}\subseteq\D_\ph\x \ker h_{\S,\io}.
\]
We need to find $Z=(Z_i)\in\X(\S;\io)$ such that $u_\ph=T\ph\o Z_r$ and $d\la_i=di_{Z_i}\al_i$.
In particular the $r$th component $Z_r$ has to be tangent to the nonlinear flag $\Sigma=(\Sigma_1,\dotsc,\Sigma_{r-1})$ in $S_r$, where $\Sigma_i:=j_i(S_i)$.

The condition that $(u_\ph,(d\la_i))$ lies in the kernel of $\Om$ is equivalent to
\begin{equation}\label{una}
\sum_{i=1}^r\int_{S_i}j_i^*(\ph^*i_{u_\ph}\om)\wedge\ga_i=0,\text{ for all }\ga_i\in\Om^{k_i-1}(S_i;\O_{S_i}),\iota_{i-1}^*\ga_i=0,
\end{equation} 
and 
\begin{equation}\label{doua}
\sum_{i=1}^r\int_{S_i}j_i^*\om(u_\ph,v_\ph)\al_i+\sum_{i=1}^r\int_{S_i}j_i^*(\ph^*i_{v_\ph}\om)\wedge\la_i=0,\text{ for all }v_\ph\in\D_\ph. 
\end{equation}
Choosing differential forms $\ga_i=0$ for $i=1,\dots,r-1$ in \eqref{una}, we get that $\int_{S_r}\ph^*i_{u_\ph}\om\wedge\ga_r=0$, 
for all $\ga_r\in\Om^{k_r-1}(S_r;\O_{S_r})$ with $\io_{r-1}^*\ga_r=0$.
This ensures that $\ph^*i_{u_\ph}\om=0$, meaning that $u_\ph$ takes values in the symplectic orthogonal to $\ph(S_r)$ in $M$. 

We assume by contradiction that $u_\ph\in\D_\ph$ is not everywhere tangent to the isotropic submanifold $\ph(S_r)\subseteq M$.
The subset $S_r\setminus\Si_{r-1}$ is dense in $S_r$ (because $\dim S_{r-1}<\dim S_r$),
so there exists $x\in S_r\setminus\Si_{r-1}$ such that $u_\ph(x)$ is not tangent to $\ph(S_r)$.
We choose another tangent vector $v_\ph\in\D_\ph$ taking values in the symplectic orthogonal to $\ph(S_r)$ in $M$,
\ie $\ph^*i_{v_\ph}\om=0$,  supported in an open neighborhood of $x$ in $S_r\setminus\Si_{r-1}$.
The isotropic embedding theorem of Weinstein \cite{W81} allows to choose $v_\ph$
such that  $\om(u_\ph,v_\ph)\in C^\oo(S_r)$ is a nonzero nowhere negative function.
Plugging $v_\ph$ in \eqref{doua}, all terms are zero except  $\int_{S_r}\om(u_\ph,v_\ph)\al_r>0$, leading to a contradiction.
Thus there exists $Z_r\in\X(S_r)$ with $u_\ph=T\ph\o Z_r$. 

The identity \eqref{doua}, rewritten with our $u_\ph=T\ph\o Z_r$ and arbitrary $v_\ph\in\D_\ph$, so $\ph^*i_{v_\ph}\om=dh$ with arbitrary $h\in C^\oo(S)$, becomes:
\begin{equation}\label{doua2}
\sum_{i=1}^r\int_{S_i}j_i^*(i_{Z_r}dh)\al_i=\sum_{i=1}^r\int_{S_i}j_i^*(dh)\wedge\la_i,\text{ for all } h\in C^\oo(S). 
\end{equation}
%For the existence of $Z_i\in\X(S_i)$ for $i=1,\dots,r-1$ with $Z=(Z_i)\in\X(\S;\io)$,
We need to show that  $Z_r\in\X(S_r)$ is tangent to all the submanifolds of the nonlinear flag $\Si_1\subseteq\dots\subseteq\Si_{r-1}\subseteq S_r$.
We do it successively, starting with $\Si_{r-1}$ and ending with $\Si_1$.

We assume by contradiction that $Z_r$ is not tangent to  $\Si_{r-1}\subseteq S_r$.
Since $\dim \Si_{r-2}<\dim \Si_{r-1}$, we can find
a point $x\in\Si_{r-1}\setminus\Si_{r-2}$ such that $Z_r(x)$ is not tangent to $\Si_{r-1}$.
We choose $h\in C^\oo(S)$ supported in a tubular neighorhod of $\Si_{r-1}$ 
with the following properties: 
$h$ vanishes on $\Si_{r-1}$ 
(so that $j_i^*dh=0$ for $i=1,\dots,r-1$)
and $i_{Z_r}dh$ restricted to $\Si_{r-1}$ is a positive bump function at $x$ that vanishes on $\Si_{r-2}$
(so that $j_i^*(i_{Z_r}dh)=0$ for $i=1,\dots,r-2$).
Inserting it in  \eqref{doua2} only three  terms survive, giving
\[
\int_{S_{r-1}}j_{r-1}^*(i_{Z_r}dh)\al_{r-1}=\int_{S_r}hd(i_{Z_r}\al_r-\la_r).
\]
The integral over $S_{r-1}$ on the left
% $c=\int_{S_{r-1}}j_{r-1}^*(i_{Z_r}dh)\al_{r-1}$ 
doesn't change when
cutting $h$ with a function supported in a tubular $\ep$-neighborhood of $\Si_{r-1}$
which is constant in an $\ep/2$-neighborhood of $\Si_{r-1}$, while the absolute value of the integral over $S_r$ on the right
can be made as small as we need by making $\ep$ small enough.
This leads to a contradiction. Thus $Z_r$ must be tangent to $\Si_{r-1}$, 
so there exists $Z_{r-1}\in\X(S_{r-1})$ which is $j_{r-1}$ related to $Z_r$.

In a similar way, step by step, we address all the remaining submanifolds $\Si_{r-2},\dots,\Si_1$ of the nonlinear flag $\Si$,
finding $Z_i\in\X(S_i)$ that are $j_i$ related to $Z_r$.
Knowing this, the identity \eqref{doua2} becomes
\[
\sum_{i=1}^r\int_{S_i}j_i^*dh\wedge(\la_i-i_{Z_i}\al_i)=0,\text{ for all }h\in C^\oo(S),
\]
which implies that $d(\la_i-i_{Z_i}\al_i)=0$ for all $i$.
Thus $(u_\ph,(d\la_i))=(T\ph\o Z_r,(di_{Z_i}\al_i))$ is the infinitesimal generator at $(\ph,(\al_i))$ for $Z=(Z_i)\in\X(\S;\io)$.
This ensures that the kernel of $\Om$ is spanned by the infinitesimal generators of the $\Diff(\S;\io)$ action.
\end{proof}

As a consequence, $\Om$ descends to a leafwise symplectic form $\bar\Om$ on the associated bundle 
\[
	(\Flag_{\S,\io}^{\wt\iso}(M),\bar\D)=(\Emb_{S_r}^{\iso}(M)\x_{\Diff(\S,\io)}\Den_\x(\S),
	\D\x_{\Diff(\S,\io)}\ker Th_{\S,\io}).
\]
The restriction of $\bar\Om$ to $(\Flag_{\S,\io,\mu}^{\wt\iso}(M),\bar\D)$
%=\Emb_{S_r}^{\iso}(M)\x_{\Diff(\S,\io)}\Den(\S)_{\io,\mu}$
is a leafwise symplectic form as well. 
Thus every leaf of $\bar\D$ is symplectic.
%and the following generalization of Theorem \ref{t1} describes them as coadjoint orbits of the Hamiltonian group:

Let $\mathcal F$ denote a leaf of the isodrastic distribution $\bar\D$ on $\Flag_{\S,\io}^{\wt\iso}(M)$, equipped with the symplectic form $\bar\Om$.
Restricting \eqref{pl}, we obtain a $\Ham(M)$ equivariant smooth map
\begin{equation}\label{bluebird}
	J\colon\mathcal F\to\ham_c(M)^*,\qquad\langle J(\N,\nu),X_f\rangle=\sum_{i=1}^r\int_{N_i}f\nu_i,
\end{equation}
where we identify $\ham_c(M)=C^\infty_0(M)$ as in \eqref{bird}.
This is a moment map for the (Hamiltonian) action of $\Ham_c(M)$ on $(\F,\bar\Om)$.
Indeed, this follows readily by combining \eqref{oioi} with the expression for the moment map in \eqref{bird}.
Moreover, $J$ is injective according to Proposition~\ref{injectiv}.
By Proposition~\ref{abo2}, the $\Ham_c(M)$ action on $\F$ is transitive.
Using Proposition~\ref{spcase}, we thus obtain the following generalization of Theorem~\ref{t1}.

\begin{theorem}\label{t2}
	The moment map $J\colon(\F,\bar\Om)\to\ham_c(M)^*$ in \eqref{bluebird} is one-to-one onto a coadjoint orbit of the Hamiltonian group $\Ham_c(M)$.
	The Kostant-Kirillov-Souriau symplectic form $\om_{\KKS}$ on the coadjoint orbit satisfies ${J}^*\om_{\KKS}=\bar\Om$.
\end{theorem}

\begin{remark}
%We know from Section~\ref{pre} that the space of weighted nonlinear flags of type $\S$ is a splitting smooth submanifold in 
%the product of weighted nonliniar Grassmannians of type $S_i$, $\prod_{i=1}^r\Gr^{\wt}_{S_i}(M)$.
%Combined with Theorem \ref{T:wtFlag.homog}, we get the splitting submanifold
	In Section~\ref{SS:manf} we have seen that the inclusion 
	\[
		\Flag_{\S,\io,\mu}^{\wt\iso}(M)\subseteq\prod_{i=1}^r\Gr^{\wt\iso}_{S_i,\mu_i}(M).
	\]
	is a splitting smooth submanifold.
	The symplectic leaf $\F$ described in Theorem~\ref{t2} is a splitting
symplectic submanifold of a product $\prod_{i=1}^r\G_i$ of symplectic
leaves described in Theorem~\ref{t1}.
         To show that this is indeed a splitting smooth submanifold, one
can proceed as in the proof of Theorem~\ref{T:wtFlag.homog}(c), using
the fact that each isodrastic leaf in $\Flag^{\iso}_{\S,\io}(M)$ is a
splitting smooth submanifold in a product of isodrastic leaves in
$\prod_{i=1}^r\Gr^{\iso}_{S_i}(M)$.
         The latter can be show by induction on the depth of the flag
using Lemma~\ref{L:GrSNMiso}.	
	\end{remark}

\begin{remark}%\cite{W90}
The leafwise symplectic form $\bar\Om$ on $\Flag_{\S,\io}^{\wt\iso}(M)$ can be seen as a Poisson structure 
whose symplectic leaves are the isodrasts $\F$ from the Theorem \ref{t2} (see Remark 3.4 in \cite{W90} about isodrasts in
the nonlinear Grassmannian of weighted Lagrangian submanifolds).
Restricting \eqref{pl}, we obtain a Poisson moment map
%The symplectic moment maps \eqref{bluebird} extend to the Poisson moment map
\[
J:\Flag_{\S,\io}^{\wt\iso}(M)\to \ham_c(M)^*,\quad \langle J(\N,\nu),X_f\rangle=\sum_{i=1}^r\int_{N_i}f\nu_i,
\]
for the Poisson action of $\Ham_c(M)$ on weighted isotropic nonlinear flags,
where we identify $\ham_c(M)=C^\infty_0(M)$ as before.
%This Poisson structure can be extended also to $\Flag_{\S,\io}^{\wt\iso}(M)$ with the same type of symplectic leaves.
\end{remark}

%\color{burgund}
\begin{example}\label{exex}
Let $M$ be a symplectic manifold that possesses isotropic submanifolds diffeomorphic to the sphere $S^r$ with $r>1$.
We use the setting of Example \ref{nest} to describe some coadjoint orbits of $\Ham_c(M)$ consisting of nested weighted spheres.
Since $H^1(S^r;\RR)=0$, by Remark~\ref{ffff} 
%the isodrastic leaves in $\Gr_{S^r}^{\iso}(M)$ are its connected components,
each connected component of $\Flag_{\S,\io,\mu}^{\wt\iso}(M)$ is a coadjoint orbit of $\Ham_c(M)$.
%If the integral of $\mu_i$ over each hemisphere of $S^i$ is $a_i$, s
Similarly to \eqref{asta} we get
\begin{equation*}
		\Flag_{\S,\io,\mu}^{\wt\iso}(M)=\left\{(\N,\nu)\in\Flag_{\S,\io}^{\wt\iso}(M)\middle|\begin{array}{c}
		\text{
		$\{\int_{N_i^+}\nu_i,\int_{N_i^-}\nu_i\}=\{a_i^+,a_i^-\}$, 
		where $N_i^+$ and $N_i^-$}\\\text{denote the connected components of  $N_i\setminus N_{i-1}$}\end{array}\right\}.
	\end{equation*}
\end{example}

\begin{example}\label{3,14}
The coadjoint orbits of $\Ham_c(\RR^2)$ consisting of pointed vortex loops, studied in \cite{CVpre}, 
are the lowest dimensional examples of coadjoint orbits of weighted nonlinear flags as described in Theorem~\ref{t2}.

Their type is $(\S,\io)$, with $\io:\{1,\dots,k\}\to S^1$, $\io(i)=t_i$ consecutive points on the circle.
%and nowhere vanishing densities $\mu_0=\sum_{i=1}^k\Ga_i\de_i$ on $S_0$ and $\mu_1$ on $S^1$ (such that $\int_{S^1}\mu_1=w$).
%A weighted nonlinear flag of this type is $((N_0,\nu_0),(N_1,\nu_1))$, with $(N_1,\nu_1)$ a weighted closed curve in $\RR^2$ (a vortex loop with total vorticity $w$) endowed with $k$ distinct points $N_0=\{x_1,\dots,x_k\}$, each one with its own weight $\Ga_i$.
We get the cohomology space 
\begin{equation*}
	H(\mathcal S,\iota)
	%=H^0(\{1,\dots,k\};\O_{\{1,\dots,k\}})\x H^1(S^1,\{t_1,\dots,t_k\};\O_{S^1})
	%=H_0(S_0;\R)\x H_0(S_1\setminus\io(S_0))
	%=H_0(\{1,\dots,k\};\mathbb R)	\x H_0\bigl(S^1\setminus\{t_1,\dots,t_k\};\mathbb R\bigr)
	\cong\R^k\x\R^k,
\end{equation*}
where the class $[\mu]\in H(\S,\io)$ is identified with its integrals
% $(\Ga_1,\dots,\Ga_k)$ of $\mu_0$ \resp $(w_1,\dots,w_k)$ of $\mu_1$ 
over connected components of $\{1,\dots,k\}$ \resp $S^1\setminus\{t_1,\dots,t_k\}$.
These are $\Ga_i=\int_{\{i\}}\mu_0$ and $w_i=\int_{t_i}^{t_{i+1}}\mu_1$, for $i=1,\dots,k$.
The $\Diff(\S,\io)$ action on $H(\S,\io)$ factorizes through an action of the dihedral group $D_{2k}$ on $\RR^k\x\RR^k$.
More precisely, one let the dihedral group act on a regular $k$-gon, with $\Ga_i$ assigned to the vertex $i$ and $w_i$ assigned to the edge $[i,i+1]$.
For generic density $\mu\in\Den_\x(\S)$, the orbit $H(\S,\io)_{[\mu]}$ consists of $2k$ elements.
The orbit is a one-point set if and only if  $\Ga_1=\dots=\Ga_k$ and $w_1=\dots=w_k$.

Let us denote the weighted flags of type $(\S,\io)$ in $\RR^2$ by $((\{x_1,\dots,x_k\},\nu_0),(C,\nu_1))$.
The area $a$ enclosed by the curve $C$ singles out one isodrast $\L_a\subset\Gr^{\iso}_{S^1}(\RR^2)$, as in Example \ref{drast}.
%, while the vorticities $(\Ga_1,\dots,\Ga_k,w_1,\dots,w_k)$ single out one element in $H(\S'\io)$. Indeed, f
From Theorem \ref{t2} follows that each connected component of $\Flag_{\S,\io,\mu}^{\wt\iso}(\RR^2)|_{\L_a}$ is a coadjoint orbit of $\Ham_c(\RR^2)$.
Using the above description of $\Diff(\S,\io)$ action on $H(\S,\io)$, we get
\begin{equation}\label{dihe}
	\Flag_{\S,\io,\mu}^{\wt\iso}(\RR^2)|_{\L_a}
	=\left\{((\{x_1,\dots,x_k\},\nu_0),(C,\nu_1))\middle|
	\begin{array}{c}
	\text{$x_i$ consecutive points in
	$C\in\L_a$, $\exists\si\in D_{2k}$}\\
	\text{s.t.  $\int_{\{x_i\}}\nu_0=\si(\Ga_i),\int_{x_i}^{x_{i+1}}\nu_1=\si(w_i)$}
	\end{array}\right\}.
\end{equation}
In particular, the invariants are the area $a$ enclosed by the loop $C$, the point vorticities $\Ga_i$, and the net vorticities $w_i=\int_{x_i}^{x_{i+1}}\nu_1$ between two consecutive points on the loop.
\end{example}

\begin{example}
Let $\L_{[a_1,a_2]}$ be a union of isodrastic leaves of Lagrangian 2-tori in $\RR^4$, as in Example \ref{drast}, with $[a_1,a_2]$ the $\SL(2,\ZZ)$ orbit of the pair of action integrals $(a_1,a_2)\in\RR^2$.
Let $S_1$ be a disjoint union of $k$ circles, $S_2=\TT^2$ the 2-torus, and $\io:S_1\to S_2$ the embedding that maps the $i$-th circle to the circle
$\{t_i\}\x S^1\subseteq\TT^2$, with consecutive points $t_1,\dots, t_k\in S^1$.
For fixed density $\mu\in\Den_\x(\S)$, we aim at describing the coadjoint orbits of $\Ham_c(\RR^4)$ that are connected components of 
$\Flag_{\S,\io,\mu}^{\wt\iso}(\RR^4)|_{\L_{[a_1,a_2]}}$.
To this end we need the isomorphism
\begin{equation*}
	H(\mathcal S,\iota)
	%=H^0(S_1;\O_{S_1})\x H^1(S_2,\io(S_1);\O_{S_2})
	%=H_0(S_1;\R)\x H_0(S_2\setminus\io(S_1))
	%=H_0(\{1,\dots,k\};\mathbb R)	\x H_0\bigl(S^1\setminus\{t_1,\dots,t_k\};\mathbb R\bigr)
	\cong\R^k\x\R^k
\end{equation*}
given by integration of $\mu_1$ over the components of $S_1$
% (denoted by $\al_1,\dots,\al_k$) 
and of $\mu_2$ over the torus surface between two successive embedded circles.
%(denoted by $\be_1,\dots,\be_k$). 
As in the previous example, the $\Diff(\S,\io)$ action on $H(\mathcal S,\iota)$ factorizes through the dihedral group. 
%and the orbit of $[\mu]$ depends on the symmetry of the $k$-tuples $\al$ and $\be$. 
One can now express these  coadjoint orbits of $\Ham_c(\RR^4)$ as in \eqref{dihe}, with the help of the dihedral group.
Thus, the invariants are: the $\SL(2,\ZZ)$ orbit of the two action integrals for the embedded 2-torus, the total weights of the $k$ isotopic loops on the torus,
and the partial weights between two such consecutive loops on the torus.
\end{example}
\color{black}

\appendix

\section{Transitive actions on associated bundles}

The manifolds and Lie groups in this appendix may be infinite dimensional and are assumed to be modelled on convenient vector spaces as in \cite{KM}.

Recall that a smooth $G$ action on $M$ is said to admit local smooth sections if every point $x_0$ in $M$ admits an open neighborhood $U$ and a smooth map $\sigma:U\to G$ such that $\sigma(x)x_0=x$, for all $x\in U$. 
Clearly, such an action is locally and infinitesimally transitive. 
Due to the lack of a general implicit function theorem, one can not expect the converse implication to hold for general Fr\'echet manifolds. 

\begin{lemma}\label{before.sec}
	Let $P\to B$ be a principal $G$-bundle endowed with the action of a Lie group $H$ on $P$ that commutes with the principal $G$ action.
	Suppose the structure group $G$ acts on another manifold $Q$, and consider the canonically induced $H$ action on the associated bundle $P\times_GQ$.
	If the $H$ action on $P$ and the $G$ action on $Q$ both admit local smooth sections, then the $H$ action on $P\times_GQ$ admits local smooth sections too.
\end{lemma}

\begin{proof}
	Suppose $\xi_0\in P\times_GQ$.
	As the canonical projection $P\times Q\to P\times_GQ$ is a locally trivial smooth bundle \cite[Theorem~37.12]{KM}, there exist an open neighborhood $U$ of $\xi_0$ as well as smooth maps $\pi:U\to P$ and $\rho:U\to Q$ such that for all $\xi\in U$ we have
	\[
		[\pi(\xi),\rho(\xi)]=\xi.
	\] 
	Put $p_0:=\pi(\xi_0)\in P$ and $q_0:=\rho(\xi_0)\in Q$.
	As the $H$ action on $P$ admits local sections, there exist an open neighborhood $V$ of $p_0$ in $P$ and a smooth map $\sigma':V\to H$ such that 
	\[
		\sigma'(p_0)=e_H\qquad\text{and}\qquad\sigma'(p)p_0=p,
	\]
	for all $p\in V$.
	As the $G$ action on $Q$ admits local sections, there exist an open neighborhood $W$ of $q_0$ in $Q$ and a smooth map $\sigma'':W\to G$ such that 
	\[
		\sigma''(q_0)=e_G\qquad\text{and}\qquad\sigma''(q)q_0=q,
	\]
	for all $q\in W$.
	Possibly replacing $U$ with a smaller neighborhood of $\xi_0$, we may assume that for all $\xi\in U$ we have $\rho(\xi)\in W$ and $\pi(\xi)\sigma''(\rho(\xi))\in V$.
	Hence, $\sigma:U\to H$,
	\[
		\sigma(\xi):=\sigma'\bigl(\pi(\xi)\sigma''(\rho(\xi))\bigr)
	\]
	is a well defined smooth map.
	For $\xi\in U$ we obtain
	\begin{multline*}
		\sigma(\xi)\xi_0
		=\sigma(\xi)[p_0,q_0]
		=[\sigma(\xi)p_0,q_0]
		=[\sigma'(\pi(\xi)\sigma''(\rho(\xi)))p_0,q_0]
\\		=[\pi(\xi)\sigma''(\rho(\xi)),q_0]
		=[\pi(\xi),\sigma''(\rho(\xi))q_0]
		=[\pi(\xi),\rho(\xi)]
		=\xi.
	\end{multline*}
	Hence, $\sigma$ is the desired local smooth section of the $H$ action on $P\times_GQ$.
\end{proof}

We will denote the fundamental vector fields of a smooth $G$ action on $M$ by
\[
	\zeta_X(x):=\tfrac\partial{\partial t}\bigl|_{t=0}\exp(tX)x,
\]
where $X\in\mathfrak g$ and $x\in M$.

\begin{lemma}\label{L:loc.sec}
	Let $G$ be a regular Lie group acting on a smooth manifold $M$.
	Suppose every point $x_0$ in $M$ admits an open neighborhood $U$ and a smooth map $\sigma:TM|_U\to\mathfrak g$ such that
	\begin{equation}\label{E:tsigma}
		\zeta_{\sigma(X)}(x)=X,
	\end{equation}
	for all $x\in U$ and $X\in T_xM$.
	Then the $G$ action on $M$ admits local smooth sections.
\end{lemma}

\begin{proof}
	We will construct a local section using an argument due to Moser \cite[Section~4]{M65}.
	Suppose $c:[0,1]\to U$ is a smooth curve.
	We seek a smooth curve $g:[0,1]\to G$ such that 
	\begin{equation}\label{E:c}
		c(t)=g(t)c(0).
	\end{equation}
	Differentiating, we obtain
	\begin{equation}\label{E:c'}
		c'(t)=\zeta_{\dot g(t)}(c(t)),
	\end{equation}
	where $\dot g(t):=\frac\partial{\partial h}\bigl|_{h=t}g(h)g(t)^{-1}$ denotes the right logarithmic derivative of $g$.

	Since $G$ is regular \cite[Definition~38.4]{KM}, there exists a unique smooth curve $g=\Evol^r\bigl(\sigma\circ c'\bigr)$ in $G$ such that $\dot g(t)=\sigma(c'(t))$ and $g(0)=e$.
	Using \eqref{E:tsigma}, we see that \eqref{E:c'} and, thus, \eqref{E:c} hold true.
	Evaluating at $t=1$, we obtain a smooth map
	\[
		s:C^\infty([0,1],U)\to G,\qquad s(c):=g(1)=\evol^r\bigl(\sigma\circ c'\bigr).
	\]
	By construction, $c(1)=s(c)c(0)$, for all smooth curves $c:[0,1]\to U$.

	To obtain a local smooth section for the $G$ action on $M$, it suffices to compose $s$ with a smooth map $U\to C^\infty([0,1],U)$, $x\mapsto c_x$ satisfying $c_x(0)=x_0$ and $c_x(1)=x$.
	The latter can readily be constructed using a chart for $M$ centered at $x_0=0$.
	Indeed, shrinking $U$ so that it becomes star shaped with center $x_0=0$ in such a chart, we may use $c_x(t):=tx$.
\end{proof}

%%%%%%%%%%%%%%%%%%%

{
\footnotesize

\bibliographystyle{new}

}
\end{document}